\documentclass{imanum}

\usepackage{amsmath,amsthm,amssymb,amsxtra,graphicx,subfigure,comment,algorithm,algpseudocode,appendix,xr,comment, multirow}
\numberwithin{equation}{section}

\newcommand{\R}{\mathbb{R}}

\newcommand{\eps}{\varepsilon}
\newcommand{\Laplace}{\Delta}

\newcommand{\rd}{\mathrm{d}}
\newcommand{\pd}{\partial}
\newcommand{\abs}[1]{\left| #1 \right|}
\newcommand{\norm}[1]{\| #1 \|}

\begin{document}

\title{Optimal control of elliptic PDEs at points}
% Short title for running heads:
\shorttitle{Optimal control of elliptic PDEs at points}

\author{%
{\sc 
Charles Brett\thanks{Corresponding author. Email: ceabrett@gmail.com. This work was supported by the UK Engineering and Physical Sciences Research Council (EPSRC) Grant EP/H023364/1.},
Andreas Dedner\thanks{Email: a.s.dedner@warwick.ac.uk}
and
Charles Elliott\thanks{Email: c.m.elliott@warwick.ac.uk}} \\[2pt]
Department of Mathematics, University of Warwick, Coventry, CV4 7AL, UK
}
% Short list of authors for running heads:
\shortauthorlist{C. Brett, A. S. Dedner and C. M. Elliott}

\maketitle

\begin{abstract}
% Body of abstract:
{We consider an elliptic optimal control problem where the objective
functional contains evaluations of the state at a finite number of
points. In particular, we use a fidelity term that encourages the
state to take certain values at these points, which means our problem
is related to ones with state constraints at points. The analysis and
numerical analysis differs from when the fidelity is in the $L^2$ norm because we need the state space to embed into the space of
continuous functions. In this paper we discretise the problem using
two different piecewise linear finite element methods. For each
discretisation we use two different approaches to prove a priori $L^2$
error estimates for the control. We discuss the differences between
these methods and approaches and present numerical results that agree
with our analytical results.}
% Keywords:
{elliptic optimal control problem; point evaluations; finite element method; error estimates}
\end{abstract}

\label{chap:pointpde}

\section{Introduction}

In this paper we study an elliptic optimal control problem with an objective functional containing the distance between the state and prescribed values at a finite number of prescribed points. This contrasts with standard elliptic optimal control problems, where typically the objective functional contains the $L^2$ distance between the state and the desired state over the whole domain. So for a bounded domain $\Omega \subset \R^n$ ($n=2$ or $3$) with boundary $\partial \Omega$ we consider the problem:
\begin{equation*}
\min \frac{1}{2} \sum_{\omega \in I} (y(\omega)-g_\omega)^2 + \frac{\nu}{2} \norm{\eta}^2_{L^2(\Omega)}
 \end{equation*}
subject to the state equation
\begin{equation} \label{eqn:ctos}
\begin{aligned}
Ay &= \eta \quad \text{ in } \Omega \\
y &= 0 \quad \text{ on } \partial \Omega
\end{aligned}
\end{equation}
and the control constraints
\[
a \leq \eta \leq b.
\]
Here $I \subset \Omega$ is a finite set of points, $\{ g_\omega \}_{\omega \in I} \subset \R$ are prescribed values at these points, $\nu > 0$ is the cost of control, $A$ is an elliptic operator, and $a,b \in \R$ with $a<b$ are lower and upper bounds for the control. We give the precise statement of the problem using function spaces in Section~\ref{sec:probformpoint}.

The motivation for the point fidelity term is that in some
applications we may only care about the state being close to given
values at certain points in the domain. Controlling the state using
a distributed norm over the whole domain yields weaker control at
points. The point fidelity term encourages the state to take certain
values at points, so our problem is closely related to one which
imposes hard constraints on the state at points. Imposing hard state
constraints can often lead to an optimal control with a very high
cost, whereas our point fidelity term allows for a compromise between
how close the state is to the prescribed values and the cost of the
control. On the other hand, we will prove later that as we increase
the relative weighting given to the point fidelity term, the solutions
of point control problems converge weakly to the solution of a problem
with point state constraints. 

In the literature there are computational results for PDE optimal control problems with objective functionals that contain point evaluations of the state. However we have not found any literature that contains a numerical analysis of such problems. The book \cite{Troltzsch2005} formulates an optimal control problem where the objective functional is the state evaluated at a point, but does not discuss numerical methods for solving it. The paper \cite{Unger2001} considers optimally controlling the cooling of steel. This problem is formulated with an objective functional that contains the temperature of the steel at a number of points (i.e.\ point evaluations of the state) as this makes the problem more tractable. The paper focuses on computational results and the numerical analysis is not considered. The medical imaging problem of electrical impedence tomography (see e.g.\ \cite{Hintermuller2008}) could be formulated as an inverse problem with a point fidelity term (but with the points on the boundary). This is because one reconstructs a conductivity based on measurements of the voltage over small regions, which could be approximated by measurements at points. In the paper \cite{Brett2013} (written by ourselves) the point fidelity term is used for the optimal control of elliptic variational inequalities. The difficulty of the nonlinear control-to-state operator means that an a posteriori error estimator is derived but a priori error estimates are not considered.

Our aim is to fill a gap in the literature by studying in detail the numerical analysis of a finite element approximation of the above point control problem, which could be considered the canonical optimal control problem with an objective functional containing point evaluations of the state. However related problems have been considered in the literature. The recent paper \cite{Gong2014} considers elliptic optimal control problems with controls at points and on other lower dimensional manifolds. The numerical analysis of these problems leads to mathematical difficulties similar to those in this paper. In particular, when the control is at points the state equation has delta functions on the right hand side, where as in our problem the adjoint equation has delta functions. In both cases this means low regularity of the state/adjoint. In the paper \cite{BrettLine} and thesis \cite{BrettThesis} theory is developed for an elliptic optimal control problem where the fidelity term is an integral along a surface of codimension 1, which is also a set of measure zero relative to the domain. In papers such as \cite{Casas2012} and \cite{Pieper2013} elliptic optimal control problems are considered where the control spaces are spaces of measures.

Regularity issues are also faced by elliptic optimal control problems with state constraints. The paper \cite{Leykekhman2013} proves error estimates for problems with state constraints at a finite number of points. Note that this paper also proves improved error estimates for graded triangulations (such triangulations are locally refined towards the singularities but have asymptotically the same number of elements for a given triangulation size), but we do not consider these. The paper \cite{Deckelnick2007} proves error estimates for the case of global (as opposed to point) state constraints, but for a state equation with Neumann boundary conditions. Parabolic optimal control problems often contain point evaluations in time of the state, but these are functions over the space domain and the technicalities of the numerical analysis are different. A review of the analysis for standard elliptic and parabolic optimal control problems can be found in \cite{Troltzsch2005} and a review of the numerical analysis can be found in \cite{Hinze2009}.

In this paper we use two different methods of discretising our problem with finite elements. The first method is to explicitly discretise the control by minimising over a space of discrete controls, leading to discrete problem $(\mathrm{M}1_h)$ (see (\ref{eqn:controlprob2point})). The second method is to implicitly discretise the control through a discrete control-to-state operator using the variational discretisation concept from \cite{Hinze2005}, leading to discrete problem $(\mathrm{M}2_h)$ (see (\ref{eqn:controlprob2point2})). We later observe that when there are no control constraints these two methods may lead to equivalent discrete problems. We are not able to prove an estimate for $(\mathrm{M}1_h)$ in dimension 3 with control constraints, which motivates us to use $(\mathrm{M}2_h)$ for our implementation despite it being less standard to solve computationally.

Next we use two different approaches to prove a priori error estimates for the $L^2(\Omega)$ error in the control for these discrete problems. The first approach (Approach~1, Section~\ref{sec:semidisc}) is inspired by the paper \cite{Casas2003} and the second approach (Approach~2, Section~\ref{sec:energy}) is inspired by the paper \cite{Deckelnick2007}. The main estimates we prove are summarised in Table~\ref{tab:sumest}, where $\eps>0$ is arbitrary. We see that Approach~2 does not offer any better error estimates than Approach~1. However we include Approach~2 because it is simpler when it applies. Numerical results confirm that the error estimates are realised for $(\mathrm{M}2_h)$.

\begin{table}
\tblcaption{The main a priori error estimates proved for $\norm{u-u_h}_{L^2(\Omega)}$.}
{%
\begin{tabular}{@{}lccc@{}}
\tblhead{Discretisation & $(\mathrm{M}1_h)$ & $(\mathrm{M}1_h)=(\mathrm{M}2_h)$ & $(\mathrm{M}2_h)$\\
Dimensions & $n=2$ & $n=2,3$ & $n=2,3$ \\
Constraints & both & $b=-a=\infty$ & both}
Approach 1 & $O(h)$ & $O(h^{2-\frac{n}{2}})$ & $O(h^{2-\frac{n}{2}})$ \\
Approach 2 & - & $O(h^{2-\frac{n}{2}-\eps})$ & $O(h^{2-\frac{n}{2}-\eps})$ \\
Numerics & - & $O(h^{2-\frac{n}{2}})$ & $O(h^{2-\frac{n}{2}})$
\lastline
\end{tabular}
}%
\label{tab:sumest}
\end{table}

In the next section we introduce some notation. In Section~\ref{sec:probformpoint} we formulate the optimal control problem precisely and prove some analytical results. In Section~\ref{sec:discpoint} we discretise using the finite element method. In Section~\ref{sec:numanalpoint} we prove a priori error estimates for the $L^2$ error in the control. In Section~\ref{sec:pointnumerics} we show numerical results.

\section{Notation}
\label{sec:notation}

% We begin by clarifying our assumptions on the state equation, then introduce its weak formulation and some suitable function spaces. These will be needed for stating the optimal control problem in Section~\ref{sec:probformpoint}. We also define the control-to-state operator and its adjoint, as they will be useful for proving results in later sections.

We begin by introducing some function spaces that are needed to formulate the optimal control problem precisely.

Let the domain $\Omega \subset \R^n$ ($n=2$ or $3$) be a bounded open set that either has a $C^{1,1}$ boundary or is convex with a polygonal (for $n=2$) or polyhedral (for $n=3$) boundary. Both $C(\bar{\Omega})$ and its subspace $C_0(\Omega)$ (of functions that are zero on $\partial \Omega$) are Banach spaces when endowed with
the supremum norm, $\norm{\cdot}_\infty$. For $n = 2$ or $3$ the Sobolev space $H^2(\Omega)$ is continuously embedded into $C(\bar{\Omega})$ (see e.g.\ \cite{Adams2003}), so $H^2(\Omega) \cap
H_0^1(\Omega)\subset C_0(\Omega)$. By different versions of the Riesz
Representation Theorem (see e.g.\ Theorems~2.14 and 6.19 in \cite{Rudin}) the dual spaces of $C(\bar{\Omega})$ and $C_0(\Omega)$ can both be identified with the
space $\mathcal{M}(\Omega)$ of real regular Borel measures on $\Omega$. In particular, for $\mu \in \mathcal{M}(\Omega)$ and $v \in C(\bar{\Omega})$ define the duality pairing 
\[
\langle \mu, v \rangle_{\mathcal{M}(\Omega)} := \int_\Omega v \rd \mu,
\]
where the integral is the Lebesgue integral with respect to $\mu$. Here $\langle \mu, v \rangle_{\mathcal{M}(\Omega)}$ abbreviates $\langle \mu, v \rangle_{\mathcal{M}(\Omega), C(\bar{\Omega})}$. Then for each $z \in C(\bar{\Omega})^*$ there exists a unique $\mu \in \mathcal{M}(\Omega)$ such that
\begin{equation}
\label{eqn:zzxz}
z(v) = \langle \mu, v \rangle_{\mathcal{M}(\Omega)} \quad \forall v \in C(\bar{\Omega}) . %\
%\langle \mu, v \rangle_{C_0(\Omega)^*} := \int_\Omega v \rd \mu, \quad \mu \in \mathcal{M}(\Omega), v \in C_0(\Omega)
\end{equation}
The same result holds for $z \in C_0(\Omega)^*$ using the same definition of $\langle \mu, v \rangle_{\mathcal{M}(\Omega)}$ but with $v \in C_0(\Omega)$. We prefer to write $\int_\Omega v \rd \mu$ but will sometimes use $\langle \mu, v \rangle_{\mathcal{M}(\Omega)}$ to simplify notation.
% \[
% \langle \mu,v \rangle_{\mathcal{M}(\Omega)} = \int_\Omega v \rd \mu \quad \mu \in \mathcal{M}(\Omega),  v \in C(\bar{\Omega}) \text{ or } C_0(\Omega).
% \]
% For a Banach space $V$, we use $\langle v^*, v \rangle_{V^*}$ to abbreviate the duality pairing $\langle v^*,v \rangle_{V^*,V} := v^*(v)$ between a functional $v^* \in V^*$ and a function $v \in V$.
Note that $\mathcal{M}(\Omega)$ is a Banach space with the norm
\[
 \norm{\mu}_{\mathcal{M}(\Omega)} :=  \abs{\mu}(\Omega) = \sup \Big \{ \int_\Omega v \rd \mu: v \in C_0(\Omega) \text{ and } \norm{v}_\infty \leq 1 \Big \},
\]
where $\abs{\mu}$ is called the total variation of $\mu$. For example, the Dirac measure centred at a point $\omega \in \Omega$, which we denote by $\delta_\omega$, is contained in $\mathcal{M}(\Omega)$ and $\norm{\delta_\omega}_{\mathcal{M}(\Omega)} = 1$. 

We will need the following embedding results for the Sobolev spaces $W^{1,s}(\Omega)$, where $V \hookrightarrow W$ denotes that $V$ is continuously embedded into $W$.
\begin{remark}
\label{rem:emb}
From \cite{Adams2003} we have that:
\begin{itemize}
\item For $s>n$,  $W^{1,s}(\Omega) \hookrightarrow C(\bar{\Omega})$;
\item For $s>\frac{2n}{n+2}$, $W^{1,s}(\Omega)\hookrightarrow L^2(\Omega)$;
\item For $s<\frac{2n}{n-2}$, $H^2(\Omega) \hookrightarrow W^{1,s}(\Omega)$.
\end{itemize}
\end{remark}

Consider the Dirichlet problem (\ref{eqn:ctos}), where the differential operator $A$ acting on a function $z:\Omega \to \R$ is defined by
\begin{equation*}
 A z = - \sum_{i,j=1}^n \pd_{x_j}(a_{ij} \pd_{x_i} z) + a_0 z
\end{equation*}
with 
\begin{align*}
&a_0 \in L^\infty(\Omega), \quad a_0(x) \geq 0 \quad \text{ for a.e. } x \in \Omega, \\ %was L^\infty(\Omega)
&a_{ij}=a_{ji} \in C^{0,1}(\bar{\Omega}), \\ %was C^1(\bar{\Omega})
&\exists \, \alpha>0 \text{ s.t.\ } \sum_{i,j=1}^n a_{ij}(x) \xi_i \xi_j \geq \alpha \abs{\xi}^2, \quad \forall x \in \Omega, \, \xi \in \R^n.
\end{align*}
In particular, $A=-\Laplace$ satisfies these assumptions. We want to work with a weak formulation of (\ref{eqn:ctos}). Define the conjugate $q'$ of $q$ to be the real number such that $\frac{1}{q}+\frac{1}{q'}=1$, and define the bilinear form $a:W_0^{1,q}(\Omega) \times W_0^{1,q'}(\Omega) \to \R$ associated to $A$ by
\begin{align*}
a(z,v) &= \sum_{i,j=1}^n \int_\Omega a_{ij}(x) \pd_{x_i}z(x) \pd_{x_j} v(x) \rd x + \int_\Omega a_0(x) z(x) v(x) \rd x, %//&= ((a_{ij}) \nabla u, \nabla v) + (a_0 u, v)
\end{align*}
where the derivatives are taken in the weak sense. By a standard result, for $\eta \in L^2(\Omega)$ there is a unique $y \in H_0^1(\Omega)$ satisfying
\begin{equation}
\label{eqn:state}
a(y,v) = (\eta,v) \quad \forall v \in H_0^1(\Omega).
\end{equation}
Here and throughout this paper $(\cdot, \cdot)$ denotes the $L^2(\Omega)$ inner product. With our assumptions on the domain $\Omega$ we have that $y \in H^2(\Omega) \cap H_0^1(\Omega)$ and 
\begin{equation*}
\norm{y}_{H^2(\Omega)} \leq C \norm{\eta}_{L^2(\Omega)}.
\end{equation*}
Here and throughout this paper $C$ is a positive constant that may vary from line to line and is independent of the variables it precedes (e.g.\ in the above equation $C$ is independent of $\eta$). For a proof of this regularity and stability result see Theorems~2.2.2.3 and 3.2.1.2 in \cite{Grisvard1985}. Since $H^2(\Omega) \hookrightarrow C(\bar{\Omega})$ we in fact have that $y \in C_0(\Omega)$ and
\begin{equation}
\label{eqn:cts}
\norm{y}_{\infty} \leq C \norm{\eta}_{L^2(\Omega)}.
\end{equation}

We define the control-to-state operator $S:L^2(\Omega) \to C_0(\Omega)$ to map $\eta \in L^2(\Omega)$ to the solution $y \in C_0(\Omega)$ of (\ref{eqn:state}). $S$ is linear, and also continuous by (\ref{eqn:cts}), so $S$ has an adjoint operator. Using (\ref{eqn:zzxz}) we can define the adjoint $S^*:\mathcal{M}(\Omega) \to L^2(\Omega)$ of $S$ by
\[
(S^*\mu, \eta) = \langle \mu, S\eta \rangle_{\mathcal{M}(\Omega)}  \quad \forall \mu \in \mathcal{M}(\Omega), \eta \in L^2(\Omega).
\]
Note that the control-to-state operator $S$ has the following characterisation.
\begin{lemma}
\label{lem:statealt}
For $\eta \in L^2(\Omega)$, $y=S\eta$ if and only if $y \in C_0(\Omega)$ satisfies
\begin{equation}
\label{eqn:stateother}
\forall q \in \Big (n, \frac{2n}{n-2} \Big): \quad  y \in W_0^{1,q}(\Omega), \quad a(y,v) = (\eta,v) \quad \forall v \in W_0^{1,q'}(\Omega).
\end{equation}
Here $(\eta,v)$ makes sense since $q \in (n, \frac{2n}{n-2})$ if and only if $q' \in (\frac{2n}{n+2},\frac{n}{n-1})$, and Remark~\ref{rem:emb} gives that for such $q'$ we have $v \in W_0^{1,q'}(\Omega) \subset L^2(\Omega)$.
\end{lemma}
\begin{proof}
Suppose $y=S\eta$ (i.e.\ it solves (\ref{eqn:state})) and take $q \in (n, \frac{2n}{n-2} )$. Since $y \in H^2(\Omega)$ we can integrate $a(y,v)$ by parts against $v \in C_c^\infty(\Omega)$ to get
\begin{equation}
\label{eqn:cper}
a(y,v) = (Ay, v) \quad \forall v \in C_c^\infty(\Omega).
\end{equation}
Then using (\ref{eqn:state}) we get
\begin{equation}
\label{eqn:xxxx}
(\eta, v) = (Ay, v) \quad \forall v \in C_c^\infty(\Omega),
\end{equation}
which implies that $Ay=\eta$ a.e.\ in $\Omega$. Moreover, it follows from (\ref{eqn:cper}) and the density of $C_c^\infty(\Omega)$ in $W_0^{1,q'}(\Omega)$ that $a(y,v)=(Ay,v)$ for all $v \in W_0^{1,q'}(\Omega)$. Combining this fact, $Ay=\eta$ a.e.\ in $\Omega$ and $v \in W_0^{1,q'}(\Omega) \subset L^2(\Omega)$ gives $a(y,v)=(\eta,v)$ for all $v \in W_0^{1,q'}(\Omega)$. % (see Theorem 1 in Section 6.3.1 in \cite{Evans}). So (\ref{eqn:xxxx}) holds for all $v \in W_0^{1,q'}(\Omega)$. 
By Remark~\ref{rem:emb} note that $y \in H^2(\Omega)\cap H_0^1(\Omega) \subset W_0^{1,q}(\Omega)$. The above arguments hold for any  $q \in (n, \frac{2n}{n-2} )$, so we have proved that $y=S\eta$ implies (\ref{eqn:stateother}) holds.

The reverse implication is also true. Since $H_0^1(\Omega) \subset W_0^{1,q'}(\Omega)$ for any $q \in (n, \frac{2n}{n-2} )$, we can test (\ref{eqn:stateother}) with any $v \in H_0^1(\Omega)$. So a solution of this must solve (\ref{eqn:state}). This completes the proof.
\end{proof}

We can use this result to prove that the adjoint operator $S^*$ can be characterised in the following way.
\begin{lemma}
\label{lem:adj}
For $\mu \in \mathcal{M}(\Omega)$, $p=S^* \mu$ if and only if $p \in L^2(\Omega)$ satisfies
\begin{equation}
\label{eqn:adj2}
\forall q' \in \Big (\frac{2n}{n+2}, \frac{n}{n-1} \Big ): \quad p \in W_0^{1,q'}(\Omega), \quad a(v,p) = \int_\Omega v \, \rd \mu \quad \forall v \in W_0^{1,q}(\Omega).
\end{equation}
Moreover,
\begin{equation}
\label{eqn:stabbb}
\norm{p}_{W_0^{1,q'}(\Omega)} \leq C(q') \norm{\mu}_{\mathcal{M}(\Omega)} \quad \forall q' \in \Big (\frac{2n}{n+2}, \frac{n}{n-1} \Big ).
\end{equation}
\end{lemma}
\begin{proof}
Suppose (\ref{eqn:adj2}) is true. Fix some $q' \in (\frac{2n}{n+2}, \frac{n}{n-1} )$ then for all $\mu \in \mathcal{M}(\Omega)$ and $\eta \in L^2(\Omega)$, testing (\ref{eqn:adj2}) with $S \eta \in W_0^{1,q}(\Omega)$ gives
\[
a(S\eta,p) =  \int_\Omega S \eta \, \rd \mu = \langle \mu , S \eta \rangle_{\mathcal{M}(\Omega)}.
\]
By the definition of $S$ we can test (\ref{eqn:stateother}) with $p \in W_0^{1,q'}(\Omega)$ to get 
\[
a(S \eta,p) = (\eta,p) = (p, \eta). 
\]
Combining these two equalities and recalling that $\mu$ and $\eta$ are arbitrary we get
\[
\langle \mu, S\eta \rangle_{\mathcal{M}(\Omega)} = (\eta,p) \quad \forall \mu \in \mathcal{M}(\Omega), \eta \in L^2(\Omega). 
\]
Comparing this to the definition of the adjoint we see $p=S^*\mu$. Since $q'$ was arbitrary we have shown (\ref{eqn:adj2}) implies $p=S^*\mu$. The uniqueness of the adjoint operator proves the reverse implication.

For the proof of the stability estimate (\ref{eqn:stabbb}) see Theorem~2 in \cite{Casas1985}. % The idea of the proof is to compute the $W^{-1,q}(\Omega)^*$ norm of $p$. This is done by making use of the estimate $\norm{S v}_{\infty} \leq C \norm{v}_{W^{-1,q}(\Omega)}$ for $v \in C_0(\Omega)$ (see e.g.\ \cite{Necas}) and the density of $C_0(\Omega)$ in $W^{-1,q}(\Omega)$.
\end{proof}
% Note that (\ref{eqn:adj2}) is well defined as there is a unique $ p \in W_0^{1,q'}(\Omega)$ solving the variational form for each $q' \in (1, \frac{n}{n-1})$ (see Theorems~1 and 2 in \cite{Casas1985}, which are proved for convex polygonal or polyhedral domains, but also apply to domains with smooth boundaries). The theorems furthermore say that for $p:=S^* \mu$,
% \begin{equation}
% \label{eqn:stabbb}
% \norm{p}_{W_0^{1,q'}(\Omega)} \leq C(q') \norm{\mu}_{\mathcal{M}(\Omega)} \quad  \forall q' \in (1, \frac{n}{n-1}).
% \end{equation}
%Since (\ref{eqn:adj2}) implies that $p \in W_0^{1,q'}(\Omega)$ for all $q' \in (1,\frac{n}{n-1})$, a Sobolev embedding result then gives that $p \in L^2(\Omega)$,  as required for a function in the range of $S^*$.
% \begin{theorem}
% % \label{thm:meas}
% Let $q>n$ (i.e.\ $q' \in (1,\frac{n}{n-1})$) and $\mu \in \mathcal{M}(\Omega)$. Then there exists a unique solution $p \in W_0^{1,q'}(\Omega)$ of problem (\ref{eqn:adj2}). Moreover there exists a constant $C(q')$ independent of $\mu$ such that

% \end{theorem}
% \begin{proof}
% This is proved in Theorem~1 in \cite{Casas1985} for polygonal or polyhedral convex domains. The same proof applies to convex domains with smooth boundaries.
% \end{proof}

\begin{remark}
\label{rem:neu}
We have assumed that the state equation is an elliptic PDE with Dirichlet boundary conditions. The theory in this paper can be adapted to elliptic PDEs with suitable Neumann boundary conditions, provided that $a(\cdot,\cdot)$ is still coercive. This is because the same regularity results hold for them and the same error estimates hold for their finite element approximations.
\end{remark}

\section{Problem formulation}
\label{sec:probformpoint}

We are now in a position to formulate the optimal control problem precisely:
\begin{equation}
\label{eqn:controlprob}
\begin{aligned}
\min \quad &J(y,\eta) := \frac{1}{2}\sum_{\omega \in I} (y(\omega)-g_\omega)^2 + \frac{\nu}{2} \norm{\eta}^2_{L^2(\Omega)} \\
\text{over } \quad &C_0(\Omega) \times L^2(\Omega) \\
\text{s.t. } \quad &y=S\eta \text{ (i.e.\ (\ref{eqn:state}) holds}) \\
\text{and } \quad &\eta \in U_{ad} :=\{\eta \in L^2(\Omega) : a \leq \eta \leq b \text{ a.e.\ in } \Omega \}.
\end{aligned}
\end{equation}
Recall that $I \subset \Omega$ is a finite set of points, $\{g_\omega\}_{\omega \in I}$ are prescribed values at these points, and $\nu>0$. We will prove results for the case that $a$ and $b$ are constant real numbers with $a<b$, and also the case of no control constraints (i.e. $b=-a=\infty$).

We can use the control-to-state operator $S$ to define the reduced objective functional $\hat{J}(\eta)=J(S\eta,\eta)$. Then it is straightforward to show that (\ref{eqn:controlprob}) is equivalent to the optimisation problem:
\begin{equation}
\label{eqn:optprobpoint}
\begin{aligned}
\min \quad &\hat{J}(\eta) = \frac{1}{2} \sum_{\omega \in I} (S\eta(\omega)-g_\omega)^2 + \frac{\nu}{2} \norm{\eta}^2_{L^2(\Omega)} \\
\text{over } \quad &\eta \in U_{ad}.
\end{aligned}
\end{equation}
This equivalence is in the sense that $u \in U_{ad}$ solves (\ref{eqn:optprobpoint}) if and only if $(Su,u)$ solves (\ref{eqn:controlprob}). It is simpler to work with the optimisation problem (\ref{eqn:optprobpoint}) for proving existence and uniqueness of a solution and deriving an optimality condition.
\begin{theorem}
\label{thm:exist}
Problem (\ref{eqn:optprobpoint}) has a unique solution $u \in U_{ad}$, hence  (\ref{eqn:controlprob}) has a unique solution $(Su,u)$.
\end{theorem}
\begin{proof}
This result follows using the same argument as is used for proving existence and uniqueness of solutions to standard optimal control problems. See e.g.\ Theorem~2.14 in \cite{Troltzsch2005} for the details.
%As $\hat{J} \geq 0$ we can construct an infimising sequence $\{ \eta_{n} \} \subset U_{ad}$ i.e.\ a sequence such that $\hat{J}(\eta_n) \to \inf_{\eta \in U_{ad}} \hat{J}(\eta)$. Note that $U_{ad}$ is a nonempty, closed, bounded and convex subset of a real reflexive Banach space, so it is weakly sequentially compact. This means there is a subsequence $\{\eta_{n_k}\}$ converging to some $u \in U_{ad}$. Since $S:L^2(\Omega) \to C_0(\Omega)$ is continuous, $\hat{J}$ is continuous. $\hat{J}$ is also convex, so it is weakly lower semicontinuous. Therefore $u$ achieves the infimum of $\hat{J}$ i.e.\ it is a minimiser of $\hat{J}$. By a contradiction argument the strict convexity of $\hat{J}$ gives that $u$ is the unique minimiser.
\end{proof}

\begin{theorem}
\label{thm:optpoint}
$u \in U_{ad}$ is a solution of (\ref{eqn:optprobpoint}) if and only if there exists a $p \in L^2(\Omega)$ such that for all $q' \in (\frac{2n}{n+2}, \frac{n}{n-1} )$, $p \in W_0^{1,q'}(\Omega)$ and
\begin{subequations} \label{eqn:system}
\begin{align}
&u \in U_{ad}, \quad (p + \nu u, v - u) \geq 0  &&\forall v \in U_{ad}, \label{eqn:systema} \\
%&a(y,v)=(u,v) &&\forall v \in H_0^1(\Omega) \\
&a(v,p)= \sum_{\omega \in I} (Su(\omega)-g_\omega)v(\omega) &&\forall v \in W_0^{1,q}(\Omega) \label{eqn:systemb} .
\end{align}
\end{subequations}
\end{theorem}
\begin{proof}
$\hat{J}:L^2(\Omega) \to \R$ has a G\^ateaux derivative $J':L^2(\Omega) \to L^2(\Omega)^*$. It is also (strictly) convex, and $U_{ad}$ is a nonempty and convex subset of a real Banach space. So by a standard result (see e.g.\ Lemma~2.21 in \cite{Troltzsch2005}) $u \in L^2(\Omega)$ is a solution of (\ref{eqn:optprobpoint}) iff
\begin{equation}
\label{eqn:vipoint}
u \in U_{ad}, \quad \langle \hat{J}'(u) ,v-u \rangle_{L^2(\Omega)^*, L^2(\Omega)} \geq 0 \quad \forall v \in U_{ad}.
\end{equation}

For notational convenience define a function $g_d \in C^\infty(\bar{\Omega})$ such that $g_d(\omega) = g_\omega$ for all $\omega \in I$; such a function could be constructed using a mollifier. Let $\mu := \sum_{\omega \in I} \delta_{\omega}$, where $\delta_{\omega}$ are Dirac measures centred at points $\omega$, so $\mu \in \mathcal{M}(\Omega)$. Since $(Su-g_d)^2 \in C(\bar{\Omega})$ we can rewrite $\hat{J}$ as
\[
\hat{J}(u) = \frac{1}{2}\int_\Omega (Su-g_d)^2 \rd \mu + \frac{\nu}{2}\norm{u}^2_{L^2(\Omega)}
\]
and use the ideas from \cite{Casas1986}. As a result our proof applies to objective functionals of this form with arbitrary $\mu \in \mathcal{M}(\Omega)$.

Calculating $\hat{J}'$ we find that (\ref{eqn:vipoint}) becomes
\[
\int_\Omega (Su-g_d)S(v-u) \rd \mu + \nu (u,v-u) \geq 0 \quad \forall v \in U_{ad}.
\]
We now show that the first term on the left hand side can be written in the form $\int_\Omega p(v-u) \rd x$, where $p$ satisfies (\ref{eqn:systemb}). 

For $u \in L^2(\Omega)$, $Su-g_d \in C(\bar{\Omega})$ and so it is measurable with respect to $\mu$. So we can define a real Borel measure $\lambda_u: \mathcal{B} \to \R$ (where $\mathcal{B}$ denotes the Borel $\sigma$-algebra of $\Omega$) by
\begin{equation} \label{eqn:lamu}
\lambda_u(A) := \int_A (Su-g_d) \rd \mu \quad \forall A \in \mathcal{B}.
\end{equation}
Since $\mu$ is regular, we can check that $\lambda_u$ is also regular.
%(see Definition 2.15 in \cite{Rudin}
So $\lambda_u$ is a real regular Borel measure (i.e.\ it belongs to $\mathcal{M}(\Omega)$) and Theorem~1.29 in \cite{Rudin} says that for $z \in C_0(\Omega)$,
\begin{equation}
\label{eqn:measrel}
\int_\Omega(Su-g_d) z \rd \mu = \int_\Omega z \rd \lambda_u.
\end{equation}
In particular, we can take $z:=S(v-u)$ to get
\[
\int_\Omega(Su-g_d) S(v-u) \rd \mu = \int_\Omega S(v-u) \rd \lambda_u = (S^*\lambda_u, v-u).
\]

Let $p:=S^*\lambda_u \in L^2(\Omega)$ then by (\ref{eqn:adj2}), for all $q' \in (\frac{2n}{n+2}, \frac{n}{n-1})$, $p \in W_0^{1,q'}(\Omega)$ and
\[
a(v,p) = \int_\Omega v \rd \lambda_u \quad \forall v \in W_0^{1,q}(\Omega).
\]
To finish, note that
\[
\int_\Omega v \rd \lambda_u = \int_\Omega(Su-g_d) v \rd \mu = \sum_{\omega \in I} (Su(\omega)-g_\omega)v(\omega).
\]
\end{proof}

% Despite having the additional unknown function $p$, these optimality conditions are actually easier to solve than conditions just involving $u$.

\begin{corollary}
\label{cor:reg}
If $u \in U_{ad}$ is a solution of (\ref{eqn:optprobpoint}) then it has the additional regularity that $u \in W^{1,q'}(\Omega)$ for all $q' \in (\frac{2n}{n+2}, \frac{n}{n-1} )$. 
\end{corollary}
\begin{proof}
Observe that (\ref{eqn:systema}) is equivalent to
\begin{equation}
\label{eqn:proj}
u(x)=\mathbb{P}_{[a,b]} \left(-\frac{1}{\nu} p(x) \right) \quad \text{ for a.e. } x \in \Omega, 
\end{equation}
where $\mathbb{P}_{[a,b]}(v) := v + \max (0, a - v ) - \max (0, v - b )$. If $v,w \in W^{1,q'}(\Omega)$ then $\max (v,w) \in W^{1,q'}(\Omega)$ (see e.g.\ \cite{Morrey1966}). So since $p \in W^{1,q'}(\Omega)$ for all $q' \in (\frac{2n}{n+2}, \frac{n}{n-1} )$, we also get this additional regularity for $u$. %that $u \in W_0^{1,q'}(\Omega)$ for all $q' \in (\frac{2n}{n+2}, \frac{n}{n-1} )$.%
\end{proof}

\subsection{Link to pointwise state constraints}
\label{sec:link}
 
We now discuss a link between the problem we consider in this paper, which penalises deviation of the state from certain values at points, and an optimal control problem with a finite number of point state constraints i.e.\ a problem that forces the state to take certain values at points. 

Consider the following problem, which is a generalisation of (\ref{eqn:controlprob}) in the case of no control constraints ($b=-a=\infty$):
\begin{equation}
\label{eqn:minweakstatecon}
\begin{aligned}
\min \quad & J^\theta_\nu(y,\eta) := \frac{1}{2} \sum_{\omega \in I} (y(\omega)-g_\omega)^2 + \nu \left( \frac{1}{2} \theta \norm{y-g_d}^2_{L^2(\Omega)} + \frac{1}{2}\norm{\eta}^2_{L^2(\Omega)} \right) \\
\text{over } \quad & C_0(\Omega) \times L^2(\Omega) \\
\text{s.t. } \quad & (\ref{eqn:state}) \text{ holds.}
\end{aligned}
\end{equation}
The modification is the addition of an optional $L^2(\Omega)$ fidelity term containing $g_d \in L^2(\Omega)$, which is weighted by $\theta \geq 0$. This allows us to relate (\ref{eqn:minweakstatecon}) to a problem with point state constraints that is considered in the literature: In the limit $\nu \to 0$ we get convergence of solutions of (\ref{eqn:minweakstatecon}) to the solution of the following problem, which can be found, for example, in \cite{Leykekhman2013}:
\begin{equation}
\label{eqn:minstatecon}
\begin{aligned}
\min \quad & J^\theta(y,\eta) := \frac{1}{2} \theta \norm{y-g_d}_{L^2(\Omega)}^2 + \frac{1}{2} \norm{\eta}^2_{L^2(\Omega)} \\
\text{over } \quad & H_0^1(\Omega) \times L^2(\Omega) \\
\text{s.t. } \quad & (\ref{eqn:state}) \text{ holds and } y(\omega)=g_\omega \text{ for } \omega \in I.
\end{aligned}
\end{equation}

\begin{theorem} \label{thm:link}
Let $(S u_\nu,u_\nu)$ solve (\ref{eqn:minweakstatecon}) for $\nu>0$ and $(S\bar{u}
,\bar{u})$ solve  (\ref{eqn:minstatecon}). Then as $\nu \to 0$,
\begin{align*}
Su_\nu &\rightharpoonup S\bar{u} \quad \text{ in } H_0^1(\Omega) \\
u_\nu &\rightharpoonup \bar{u} \quad \quad \text{in } L^2(\Omega).
\end{align*}
\end{theorem}
\begin{proof}
First note that there exists a function $\hat{u} \in L^2(\Omega)$ such that $S \hat{u}(\omega)=g_\omega$ for all $\omega \in I$ (see Lemma 1 in \cite{Leykekhman2013}), so $J_{\nu}^\theta( S u_\nu,u_\nu) \leq \nu J^\theta(S \hat{u},\hat{u})$. For all $\nu>0$, $(S \hat{u},\hat{u})$ is feasible for (\ref{eqn:minweakstatecon}) so
\begin{equation}
\label{eqn:unifeps}
\frac{\nu}{2} \norm{u_\nu}^2_{L^2(\Omega)} \leq J_{\nu}^\theta( S u_\nu,u_\nu) \leq \nu J^\theta(S \hat{u},\hat{u}) \leq C \nu
\end{equation}
with $C$ independent of $\nu$. So $u_\nu$ is uniformly bounded with respect to $\nu$ in $L^2(\Omega)$, which means for every sequence $\nu_k \to 0$ there exists a sequence $u_{\nu_k} \rightharpoonup \tilde{u}$ in $L^2(\Omega)$. Moreover (\ref{eqn:unifeps}) and the stability result
\[
\norm{S u_{\nu_k}}_{H_0^1(\Omega)} \leq C \norm{u_{\nu_k}}_{L^2(\Omega)}
\]
with $C$ independent of $u_{\nu_k}$ allows us to find a further subsequence, which we also denote by $\{\nu_k\}$, such that $S u_{\nu_k} \rightharpoonup \tilde{y}$ in $H_0^1(\Omega)$. Then taking the limit in (\ref{eqn:state}) we see that $\tilde{y}=S \tilde{u}$. To complete the proof we need to show that $\tilde{u}=\bar{u}$, which we do by showing that $(S \tilde{u},\tilde{u})$ is feasible for (\ref{eqn:minstatecon}) and that $J^\theta(S \tilde{u},\tilde{u}) \leq J^\theta(S\bar{u},\bar{u})$. 

Note that the same reasoning as for (\ref{eqn:unifeps}) gives $\frac{1}{\nu} \sum_{\omega \in I} (S u_\nu(\omega)-g_\omega)^2 \leq C$ independently of $\nu$. Therefore we must have $S u_\nu(\omega) \to g_\omega$. So $S \tilde{u}(\omega)=g_\omega$ for $\omega \in I$ and $(S \tilde{u},\tilde{u})$ is feasible for (\ref{eqn:minstatecon}).

The weak lower semicontinuity of $J^\theta$ and $J^\theta(S u_{\nu_k},u_{\nu_k}) \leq \frac{J^\theta_{\nu_k} (S u_{\nu_k},u_{\nu_k})}{\nu_k}$ implies
\[
J^\theta(S \tilde{u},\tilde{u}) \leq \liminf_{k \to \infty} J^\theta(S u_{\nu_k},u_{\nu_k}) \leq \liminf_{k \to \infty} \frac{J^\theta_{\nu_k} (S u_{\nu_k},u_{\nu_k})}{\nu_k}.
\]
Also the optimality of $(S u_{\nu_k},u_{\nu_k})$ for (\ref{eqn:minweakstatecon}) and $\frac{J^\theta_{\nu_k}(S \bar{u},\bar{u})}{\nu_k} = J^\theta(S \bar{u},\bar{u})$ implies
\[
\liminf_{k \to \infty} \frac{J^\theta_{\nu_k}(S u_{\nu_k},u_{\nu_k})}{\nu_k} \leq \liminf_{k \to \infty} \frac{J^\theta_{\nu_k}(S\bar{u} ,\bar{u})}{\nu_k} = J^\theta(S \bar{u},\bar{u}).
\]
Combining these we get
\[
J^\theta(S \tilde{u},\tilde{u}) \leq J^\theta(S \bar{u},\bar{u}),
\]
so we have proved the result.
\end{proof}

\section{Discretisation}
\label{sec:discpoint}

In this section we discretise the state equation using a finite element method and use this to formulate two different discrete problems. We then derive discrete optimality conditions for each problem. % that can be solved numerically.
 
We now make slightly stronger assumptions on $\Omega$ than were necessary for the problem formulation and analysis in the previous section. From now onwards assume that $\Omega$ is convex with a $C^2$ boundary. %In addition to being assumed by some of the results we cite, 
The assumption of convexity simplifies the presentation since then the finite element space for the state (defined shortly) is a subset of $C_0(\Omega)$. Note that if the state equation had Neumann boundary conditions (see Remark~\ref{rem:neu}) then nonconvex domains would not cause this complication. %Some of the results we cite are proved for convex polygonal domains, but they also hold for convex smooth domains. 
A $C^2$ boundary is sufficiently smooth that for $2 \leq s < \infty$,
\begin{equation}
\label{eqn:lasd}
\norm{S \eta}_{W^{2,s}(\Omega)} \leq C(s) \norm{\eta}_{L^s(\Omega)} \quad \forall \eta \in L^s(\Omega)
\end{equation}
(see e.g.\ Theorems~9.14 and 9.15 in \cite{Trudinger}). 

We can take a family of polygonal approximations $\Omega_h \subset \Omega$ such that the vertices of $\partial \Omega_h$ lie on $\partial \Omega$ and $\abs{\Omega \setminus \Omega_h} \leq C h^2$. On each $\Omega_h$ we can construct a conforming triangulation $T_h$ of triangles or tetrahedra $T$ with maximum diameter $h:=\max_{T \in T_h} h(T)$, where $h(T)$ is the diameter of an element $T$. Additionally suppose that the family of triangulations are conforming and quasi-uniform i.e.\ there exists a constant $C$ such that
\[
\frac{h(T)}{\rho(T)} \leq C \quad \forall T \in T_h,
\]
where $\rho(T)$ is the radius of the largest ball contained in $T$, and there exists a constant $C$ such that
\[
\frac{h}{h(T)} \leq C \quad \forall T \in T_h
\]
(see e.g.\ Chapter 3 in \cite{Ciarlet1978}). We can define the following family of discrete spaces of piecewise linear globally continuous finite elements which vanish on the boundary:
\begin{align*}
V_h := \{v_h \in C_0(\Omega) : v_h|_T \in P_1(T) \text{ for all } T \in T_h \text{ and } v_h|_{\Omega \setminus \Omega_h} = 0 \}.
\end{align*}
Here $P_1(T)$ is the set of affine functions over $T$. Our motivation for using this finite element space (rather than, for example, a space of piecewise constant finite elements) is that it is a subspace of $C_0(\Omega)$. 

We also construct a family of triangulations $T^\sigma$ of triangles or tetrahedra with maximum element diameter $\sigma$. We allow elements on the boundary to have one curved face, and assume that $T^\sigma$ is conforming and shape regular (as we did for $T_h$). Note that the family of triangulations $T^\sigma$ potentially has nothing in common with $T_h$. We can now define the following discrete space $U_{ad,\sigma}$ for the control:
\begin{align*}
U_\sigma &:= \{u_\sigma \in C(\bar{\Omega}) : u_\sigma|_T \in P_1(T) \text{ for all } T \in T^\sigma \}, \\
U_{ad,\sigma} &:= \{u_{\sigma} \in U_\sigma : a \leq u_\sigma \leq b \}. % \subset U_{ad}.
\end{align*}
This is a space of piecewise linear globally continuous finite elements (as was $V_h$) with $U_{ad,\sigma} \subset U_{ad}$, however we do not require the functions to vanish at the boundary. Recall from Corollary~\ref{cor:reg} that $u \in W^{1,q'}(\Omega)$ for all $q' \in (\frac{2n}{n+2}, \frac{n}{n-1} )$, and piecewise constant finite elements approximate such functions almost as well as piecewise linear finite elements. However we define $U_{ad,\sigma}$ to use piecewise linear finite elements as we want to allow taking the same discrete space for the control and state. This can simplify implementations.
 
For $U_{\sigma}$ the following approximation property holds: There exists an interpolation operator $\Pi_\sigma:W^{l,p}(\Omega) \to U_\sigma$ ($l=1,2$; $1\leq p< \infty$) 
such that
\begin{align}
\norm{v-\Pi_\sigma v}_{W^{m,p}(\Omega)} &\leq C\sigma^{1-m} \norm{v}_{W^{1,p}(\Omega)}, \quad m=0,1. \label{eqn:approx} %\\
% \norm{v-\Pi_hv}_{W^{1,\infty}(\Omega)} &\leq C h^{1-\frac{n}{p}} \norm{v}_{W^{2,p}(\Omega)}, \quad p>n. \label{eqn:approx2}
\end{align}
Such an interpolation operator can be defined as in \cite{SCOTT1990}. It also has the property that $v \in U_{ad}$ implies $\Pi_\sigma v \in U_{ad}$. 

We now introduce discrete approximations of $S$ and $S^*$. Define $S_h:L^2(\Omega) \to C_0(\Omega)$ by $S_h \eta=y_h$, where $y_h$ satisfies
\begin{equation} \label{eqn:statedisc}
y_h \in V_h, \quad a(y_h,v_h) = (\eta,v_h) \quad \forall v_h \in V_h.
\end{equation}
It is a standard result that this problem has a unique solution. We now prove some estimates for $S_h$ that will be useful for the numerical analysis.

\begin{lemma}
\label{lem:prelim}
For $\eta \in L^s(\Omega)$ and $2\leq s<\infty$,
\begin{equation}
%\label{eqn:errinf}
\norm{S\eta-S_h \eta}_\infty \leq C(s)h^{2-\frac{n}{s}} \norm{\eta}_{L^s(\Omega)}, \quad n=2,3.
\end{equation}
\end{lemma}
\begin{proof}
First we will recall some results from the literature that hold under the assumptions we make in this paper. By (34) in \cite{Leykekhman2013} we have that
\[
\norm{S v-S_h v}_{L^s(\Omega)} \leq C(s) h^2 \norm{S v}_{W^{2,s}(\Omega)} \quad \forall v \in L^s(\Omega).
\] 
This was originally proved for $n=2$ on p438 in \cite{ScottPaper}. Applying an inverse inequality on each element of the triangulation %and noting that function in $V_h$ are zero in $\Omega \setminus \Omega_h$ 
gives that
\begin{equation}
\label{eqn:invineq}
\norm{v_h}_{L^\infty(\Omega_h)} \leq C(s)h^{-\frac{n}{s}}\norm{v_h}_{L^s(\Omega)} \quad \forall v_h \in V_h 
\end{equation}
(see e.g.\ \cite{Ciarlet1978}). Similarly, for the piecewise linear interpolation operator $I_h:C_0(\Omega) \to V_h$ and $r \in [1,\infty]$ we have
\[
\norm{v -I_hv}_{L^r(\Omega_h)} \leq C(s) h^{2 + \frac{1}{r}-\frac{1}{s}} \norm{v}_{W^{2,s}(\Omega_h)} \quad \forall v \in W^{2,s}(\Omega).
\]
(see e.g.\ Theorem~3.1.5 in \cite{Ciarlet1978}).

Combining these results we get that
\begin{align*}
\norm{S\eta-S_h \eta}_{L^\infty(\Omega_h)} &\leq \norm{S \eta - I_h S \eta}_{L^\infty(\Omega_h)} + \norm{I_h S\eta - S_h \eta}_{L^\infty(\Omega_h)}, \\
&\leq C(s)(h^{2-\frac{n}{s}} \norm{S \eta}_{W^{2,s}(\Omega)} + h^{-\frac{n}{s}} \norm{I_h S\eta - S_h \eta}_{L^s(\Omega_h)}) \\
&\leq C(s) h^{-\frac{n}{s}} ( h^2 \norm{S \eta}_{W^{2,s}(\Omega)} + \norm{I_h S\eta - S \eta}_{L^s(\Omega_h)} + \norm{S \eta - S_h \eta}_{L^s(\Omega_h)}) \\
&\leq C(s) h^{2-\frac{n}{s}}(\norm{S \eta}_{W^{2,s}(\Omega)} + \norm{\eta}_{L^{s}(\Omega)}) \\
&\leq C(s) h^{2-\frac{n}{s}} \norm{\eta}_{L^{s}(\Omega)}.
\end{align*}

We now need to prove a supremum norm error estimate for the skin $\Omega \setminus \Omega_h$. By Theorem 4.12 Part II in \cite{Adams2003}:
\begin{itemize}
\item If $s\geq n$ then $W^{2,s}(\Omega) \hookrightarrow C^{0,\lambda}(\bar{\Omega})$ for $0<\lambda<1$.
\item If $\frac{n}{2}<s<n$ then $W^{2,s}(\Omega) \hookrightarrow C^{0,\lambda}(\bar{\Omega})$ for $0<\lambda\leq2-\frac{n}{s}$.
\end{itemize}
Let 
\[
\bar{\lambda}(s):=
\begin{cases}
1-\frac{n}{2s} & s \geq n, \\
2-\frac{n}{s} & \frac{n}{2} \leq s<n, \\
\end{cases}
\]
and observe that for $x_1 \in \Omega \setminus \Omega_h$ we have
\[
\inf_{x_2 \in \partial \Omega} \abs{ S \eta(x_1) - S \eta(x_2) } \leq C(s) \inf_{x_2 \in \partial \Omega} \abs{x_1-x_2}^{\bar{\lambda}(s)}.
\]
From the smoothness of the domain it follows that 
\[
\inf_{x_2 \in \partial \Omega} \abs{x_1-x_2} \leq C h^2 \quad \forall x_1 \in \Omega \setminus \Omega_h
\]
and for $2 \leq s< \infty$ we have $h^{2\bar{\lambda}(s)} \leq Ch^{2-\frac{n}{s}}$ for sufficiently small $h$. Using this and $S \eta |_{\partial \Omega} =0$ we get
\[
\abs{ S \eta(x_1) } \leq C(s) h^{2-\frac{n}{s}} \quad \forall x_1 \in \Omega \setminus \Omega_h.
\]
Hence
\[
\norm{S\eta-S_h \eta}_\infty \leq \max( \norm{S \eta-S_h \eta}_{L^\infty(\Omega_h)}, \norm{S \eta-S_h \eta}_{L^\infty(\Omega \setminus \Omega_h)}) \leq C(s) h^{2-\frac{n}{s}} \norm{\eta}_{L^{s}(\Omega)}. 
\]
\end{proof}

\begin{corollary}
\label{cor:prelim}
For $\eta \in L^2(\Omega)$,
\begin{equation}
\label{eqn:errinf}
\norm{S\eta-S_h \eta}_\infty \leq Ch^{2-\frac{n}{2}} \norm{\eta}_{L^2(\Omega)}, \quad n=2,3.
\end{equation}
For $\eta \in W^{1,q'}(\Omega)$ with $q' \in (\frac{2n}{n+2}, \frac{n}{n-1} )$,
\begin{equation}
\label{eqn:precise}
\norm{S \eta-S_h \eta}_\infty \leq C(q') h^{3-\frac{n}{q'}} \norm{\eta}_{W^{1,q'}(\Omega)} \quad n=2,3.
\end{equation}
\end{corollary}
\begin{proof}
The first estimate follows by taking $s=2$ in Lemma~\ref{lem:prelim}. The other estimate follow by combining the lemma with Sobolev embedding results. In particular, if  $\eta \in W^{1,q'}(\Omega)$ with $q' \in (\frac{2n}{n+2}, \frac{n}{n-1} )$ then $W^{1,q'}(\Omega) \hookrightarrow L^{s}(\Omega)$ with $s = \frac{nq'}{n-q'} \geq 2$. So
\[
C(s)h^{2-\frac{n}{s}}\norm{\eta}_{L^{s}(\Omega)} \leq C(q')  h^{3-\frac{n}{q'}} \norm{\eta}_{W^{1,q'}(\Omega)},
\]
which proves the second estimate. Note that this estimate is proved in a similar way in Theorem~3 in \cite{Leykekhman2013}.
\end{proof}

We will use (\ref{eqn:errinf}) in Section~\ref{sec:semidisc} and (\ref{eqn:precise}) in Section~\ref{sec:energy} to prove $L^2(\Omega)$ error estimates for the point optimal control problem.

Since $S_h$ is continuous (which follows from (\ref{eqn:errinf})) and linear it has an adjoint operator $S^*_h:\mathcal{M}(\Omega) \to L^2(\Omega)$. Note that the same calculation as in Lemma~\ref{lem:adj} gives that $p_h = S^*_h \mu$ if and only if $p_h$ satisfies
\begin{equation} \label{eqn:adjdisc}
p_h \in V_h, \quad a(v_h,p_h) = \int_\Omega v_h \rd \mu \quad \forall v_h \in V_h.
\end{equation}
We have the following error estimate for $S_h^*$, which we will use in Section~\ref{sec:semidisc}: For $\mu \in \mathcal{M}(\Omega)$,
\begin{equation}
\label{eqn:err2}
\norm{S^*\mu - S^*_h \mu}_{L^2(\Omega)} \leq C h^{2-\frac{n}{2}} \norm{\mu}_{\mathcal{M}(\Omega)},
\end{equation}
with $C$ independent of $\mu$ and $h$. This follows by noting that for any $v \in L^2(\Omega)$,
\[
(S^*\mu-S_h^*\mu,v) = \langle \mu, S v-S_h v \rangle_{\mathcal{M}(\Omega)} \leq \norm{\mu}_{\mathcal{M}(\Omega)} \norm{Sv-S_hv}_\infty.
\]
Then using (\ref{eqn:errinf}) gives the result. The estimate (\ref{eqn:err2}) was originally proved for convex polygonal domains in Theorem~3 in \cite{Casas1985}, and related theory is developed in \cite{Scott1973}.

\begin{remark}
The estimates in Lemma~\ref{lem:prelim} and Corollary~\ref{cor:prelim} still hold if $S$ and $S_h$ are appropriately defined control-to-state operators corresponding to an elliptic PDE with Neumann boundary conditions. 
%For example, Lemma 3.4 in \cite{Deckelnick2008} gives that for $\eta \in W^{1,q'}(\Omega)$ with $q' \in (\frac{2n}{n+2}, \frac{n}{n-1} )$,
%\begin{equation}
%\norm{S \eta-S_h \eta}_\infty \leq C(q') h^{3-\frac{n}{q'}} \abs{ \log h } \norm{\eta}_{W^{1,q'}(\Omega)},
%\end{equation}
%which is analogous to (\ref{eqn:precise}).
\end{remark}

\subsection{Discrete problems}

We are now ready to introduce the two discrete problems that we consider in our numerical analysis.

Define the discrete reduced objective functional $\hat{J}_h: L^2(\Omega) \to \R$ by
\begin{equation*}
\hat{J}_h(\eta) = J(S_h \eta, \eta) = \frac{1}{2} \sum_{\omega \in I} (S_h \eta(\omega)-g_\omega)^2 + \frac{\nu}{2} \norm{\eta}^2_{L^2(\Omega)}.
\end{equation*}
Then the first discrete problem we consider is ($\mathrm{M}1_h$):
\begin{equation}
\label{eqn:controlprob2point} 
\min \hat{J}_h(\eta_\sigma) \text{ over } \eta_\sigma \in U_{ad,\sigma}.
\end{equation}
 
\begin{proposition} There is a unique solution $u_{\sigma,h} \in U_{ad,\sigma}$ to $(\mathrm{M}1_h)$ (see (\ref{eqn:controlprob2point})). Moreover, $u_{\sigma,h} \in U_{ad,\sigma}$ is a solution of $(\mathrm{M}1_h)$ if and only if there exists $p_h \in V_h$ such that
\begin{subequations}
\begin{align}
&u_{\sigma,h} \in U_{ad,\sigma}, \quad (p_h + \nu u_{\sigma,h},v_\sigma-u_{\sigma,h}) \geq 0 && \forall v_\sigma \in U_{ad,\sigma} \label{eqn:system2apoint}  \\
%&a(y_h,v_h) = (u_{\sigma,h},v_h) \quad \forall v_h \in V_h \\
&a(v_h,p_h) = \sum_{\omega \in I} (S_h u_{\sigma,h}(\omega)-g_\omega)v_h(\omega) && \forall v_h \in V_h. \label{eqn:system2bpoint} 
\end{align}
\end{subequations}
\end{proposition}
\begin{proof}
The proof follows from the same considerations as in Theorems \ref{thm:exist} and \ref{thm:optpoint}. Note that $p_h = S_h^* \lambda_{h, u_{\sigma,h}}$ where for $\eta \in L^2(\Omega)$ we define $\lambda_{h,\eta} \in \mathcal{M}(\Omega)$ by
\begin{equation}
\label{eqn:meassy}
\lambda_{h,\eta}(A) = \int_A (S_h \eta - g_d) \rd \mu \quad \forall A \in \mathcal{B}
\end{equation}
with $\mu = \sum_{\omega \in I} \delta_\omega$.
\end{proof}

We refer to $(\mathrm{M}1_h)$ as the explicitly discretised problem as we make the control belong to a space of discrete functions.
 
Alternatively we could use the variational discretisation concept from \cite{Hinze2005} and leave the control in the infinite dimensional space $U_{ad}$. This leads to the potentially different (see Remark~\ref{rem:potdiff}) discrete problem ($\mathrm{M}2_h$):
\begin{equation}
\label{eqn:controlprob2point2} 
\min \hat{J}_h(\eta) \text{ over } \eta \in U_{ad}.
\end{equation}

\begin{proposition} There is a unique solution $u_h \in U_{ad}$ to $(\mathrm{M}2_h)$ (see (\ref{eqn:controlprob2point2})). Moreover, $u_h \in L^2(\Omega)$ is a solution of $(\mathrm{M}2_h)$ if and only if there exists $p_h \in V_h$ such that
\begin{subequations} \label{eqn:system22}
\begin{align}
&u_h \in U_{ad}, \quad (p_h + \nu u_h,v-u_h) \geq 0 && \forall v \in U_{ad} \label{eqn:system2apoint2}  \\
%&a(y_h,v_h) = (u_{\sigma,h},v_h) \quad \forall v_h \in V_h \\
&a(v_h,p_h) = \sum_{\omega \in I} (S_h u_h(\omega)-g_\omega)v_h(\omega) && \forall v_h \in V_h. \label{eqn:system2bpoint2} 
\end{align}
\end{subequations}
\end{proposition}
\begin{proof}
The proof also follows from the same considerations as in Theorems \ref{thm:exist} and \ref{thm:optpoint}.
\end{proof}

A priori we only know that $u_h$ belongs to $U_{ad}$. However observe that (\ref{eqn:system2apoint2}) can be expressed using the pointwise projection operator $\mathbb{P}_{[a,b]}$ from (\ref{eqn:proj}) as
\begin{equation*}
u_h = \mathbb{P}_{[a,b]} \Big ( -\frac{1}{\nu}{p_h} \Big ).
\end{equation*}
So (\ref{eqn:system2apoint2}) has a simpler form than (\ref{eqn:system2apoint}), which is an $L^2(\Omega)$ projection onto a discrete space. This means $u_h$ inherits a piecewise linear structure from $p_h \in V_h$, but observe that $u_h$ does not necessarily belong to $V_h$ due to the control constraints. We refer to this as an implicit discretisation; we are not requiring $u_h$ to be a piecewise linear function, but it gains this property indirectly through the discretisation of the state. Even though $u_h$ does not necessarily belong to $V_h$, this problem can be solved computationally. We will elaborate on this in Section~\ref{sec:nummeth}.

\begin{remark}
\label{rem:potdiff}
The motivation for the implicitly discretised problem $(\mathrm{M}2_h)$ is that it allows a better approximation of the set where the control constraints are active (indicated in Figure \ref{fig:proj}), likely leading to a smaller error. For a more thorough explanation see \cite{Hinze2005}.
\end{remark}

\begin{remark}
\label{rem:discrep}
Note that if there are no active control constraints (e.g.\ if $b=-a=\infty$) and $V_h \subset U_\sigma$, then $(\mathrm{M}1_h)$ and $(\mathrm{M}2_h)$ are equivalent. In order for $V_h \subset U_\sigma$ we need $``T^\sigma \subset T_h"$. By this we mean that each element of $T^\sigma$ is contained in either a single element of $T_h$ or the skin $\Omega \setminus \Omega_h$.
\end{remark}

\begin{figure}[ht]
\centering
\includegraphics[width=0.35\textwidth]{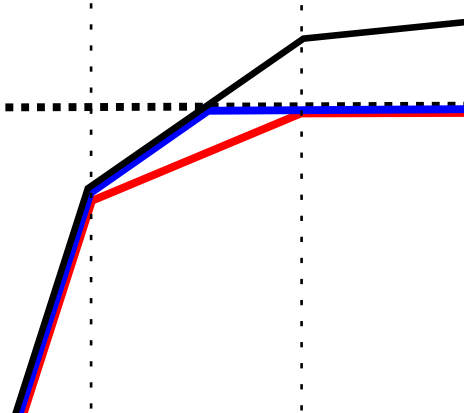}
\caption[Illustration of projections.]{An illustration in 1D of how $u_h$ is determined by $p_h$ (black line) when the discrete space for the control and state are the same and $\nu=1$. The horizontal dashed line is the value of $b$ and the vertical dashed lines marks the boundary between elements. The blue line is $u_h$ calculated from $p_h$ using (\ref{eqn:system2apoint}) and the red line is using (\ref{eqn:system2apoint2}). Assuming the $p_h$ are similar and good approximations of $p$ for both $(\mathrm{M}1_h)$ and $(\mathrm{M}2_h)$, this suggests that $(\mathrm{M}2_h)$ will give a better approximation of $u$.
}
\label{fig:proj}
\end{figure}

% In the case of active control constraints we are able to prove better error estimates for (\ref{eqn:controlprob2point2}) (though our error estimates may not be optimal for (\ref{eqn:controlprob2point})). This, and the fact that we can only find explicit solutions and reliably verify error estimates in the case of no control constraints, motivates us to only solve (\ref{eqn:controlprob2point2}) numerically.

% We will only prove error estimates for problem (\ref{eqn:controlprob2point}), but solve problem (\ref{eqn:controlprob2point2}) numerically. This is justified because in the case that we have no control constraints ($b=-a=\infty$) and $U_{ad,\sigma} = V_h$ (i.e.\ the triangulations conincide), (\ref{eqn:controlprob2point}) and (\ref{eqn:controlprob2point2}) are equivalent. This is the only case for which we can find explicit solutions and reliably verify rates of convergence numerically.

% In the case of control constraints there will be a discrepency between the problem that we perform our numerical analysis on and the one we solve numerically. \cite{Hinze2005} suggests that the problem we solve numerically will have the same or better error estimates than those we prove. We shall see that we in fact observe the same error estimates, supporting that our proved results for (\ref{eqn:controlprob2point}) are optimal, depsite us having solved (\ref{eqn:controlprob2point2}) numerically. We feel these numerical experiments are sufficient for the case of control constraints as they are less reliable due to us not having an explicit solution to compute the error against.
\section{Numerical analysis}
\label{sec:numanalpoint}

We now prove $L^2(\Omega)$ error estimates between the solution of the continuous problem (\ref{eqn:optprobpoint}) and the two discrete problems $(\mathrm{M}1_h)$ and $(\mathrm{M}2_h)$ (see (\ref{eqn:controlprob2point}) and (\ref{eqn:controlprob2point2})). 
%To prove these we use a different method for each of the two discrete problems, although a by-product of the first method actually provides a numerically realised error estimate for $(\mathrm{M}2_h)$ too.
We use two different approaches for this numerical analysis. Approach~1 in the next section allows us to prove error estimates for the two discrete problems in most (but not all) the cases we would like. Approach~2 in Section~\ref{sec:energy} only reproduces some of these error estimates, however it is simpler.
%and in addition provides error estimates for $(\mathrm{M}2_h)$ without the need for further assumptions.

\subsection{Approach 1}
\label{sec:semidisc}

This error analysis is based on \cite{Casas2003}, where an a priori $L^2(\Omega)$ error estimate is proved for the standard optimal control problem which has an $L^2(\Omega)$ fidelity term. The approach allows us to prove $L^2(\Omega)$ error estimates for both $(\mathrm{M}1_h)$ and $(\mathrm{M}2_h)$. The only estimates it does not give are ones for $(\mathrm{M}1_h)$ when $n=3$ (but we are not able to prove these using Approach~2 either). In particular we will get the following results.

\begin{theorem}
\label{thm:main}
Assume $n=2$. Let $u$ solve (\ref{eqn:optprobpoint}) and $u_{\sigma,h}$ solve $(\mathrm{M}1_h)$ (see (\ref{eqn:controlprob2point})). Then
\begin{equation*}
\norm{u-u_{\sigma,h}}_{L^2(\Omega)} \leq C (\sqrt{\sigma} + h)
\end{equation*}
with $C$ independent of $\sigma$ and $h$.
\end{theorem}

\begin{theorem} \label{thm:main22}
Assume $n=2$ or $3$. Let $u$ solve (\ref{eqn:optprobpoint}) and $u_h$ solve $(\mathrm{M}2_h)$ (see (\ref{eqn:controlprob2point2})). Then
\begin{equation*}
\norm{u-u_h}_{L^2(\Omega)} \leq C h^{{2-\frac{n}{2}}}
\end{equation*}
with $C$ independent of $\sigma$ and $h$.
\end{theorem}  

\begin{corollary}
\label{cor:thm12}
Assume $n=2$ or $3$, there are no active control constraints (e.g.\ $b=-a=\infty$), and that $V_h \subset U_\sigma$. Let $u$ solve (\ref{eqn:optprobpoint}) and $u_{\sigma,h}$ solve $(\mathrm{M}1_h)$ (see (\ref{eqn:controlprob2point})). Then
\begin{equation*}
\norm{u-u_{\sigma,h}}_{L^2(\Omega)} \leq C h^{{2-\frac{n}{2}}}
\end{equation*}
with $C$ independent of $\sigma$ and $h$.
\end{corollary}
\begin{proof}
This result follows from the equivalence between $(\mathrm{M}1_h)$ and $(\mathrm{M}2_h)$ that is highlighted in Remark~\ref{rem:discrep}.
\end{proof}

Note that these results suggest $(\mathrm{M}2_h)$ is the preferred discretisation. In particular, we can only prove an error estimate in the case of $n=3$ with control constraints for $(\mathrm{M}2_h)$. Also the error estimate in the case of $n=2$ with control constraints is better for  $(\mathrm{M}2_h)$.

The idea of the approach is to consider the error caused by the discretisation of the control and state separately, then combine them. This approach only needs the weak supremum norm error estimate for the state equation (where as a stronger one is needed for Approach~2 in Section~\ref{sec:energy}), but it does require an error estimate for the adjoint of the control-to-state operator. An advantage of this approach is that it can give insight into the best choice of triangulations for the control and state, which are not necessarily the same.

To begin we split the error as
\begin{equation}
\label{eqn:error}
\norm{u-u_{\sigma,h}}_{L^2(\Omega)} \leq \norm{u-u_\sigma}_{L^2(\Omega)} + \norm{u_\sigma - u_{\sigma,h}}_{L^2(\Omega)},
\end{equation} 
where $u_\sigma$ solves the semi discrete control problem
\begin{equation}
\label{eqn:controlprob3} 
\min \hat{J}(\eta_\sigma) \text{ over } \eta_\sigma \in U_{ad,\sigma}.
\end{equation}
\begin{proposition}
\label{prop:tmp}
There is a unique solution $u_\sigma \in U_{ad,\sigma}$ to (\ref{eqn:controlprob3}). Moreover, $u_\sigma \in U_{ad,\sigma}$ is a solution of (\ref{eqn:controlprob3}) if and only if there exists a $p_\sigma \in L^2(\Omega)$ such that for all $q' \in (\frac{2n}{n+2}, \frac{n}{n-1} )$, $p_\sigma \in W_0^{1,q'}(\Omega)$ and
\begin{subequations}
\label{eqn:system3}
\begin{align}
&u \in U_{ad,\sigma} \quad (p_\sigma + \nu u_\sigma,v_\sigma-u_\sigma) \geq 0 \quad &&\forall v_\sigma \in U_{ad,\sigma} \label{eqn:system3a} \\
%&a(y,v) = (u_\sigma,v) \quad \forall v \in H_0^1(\Omega) \\
&a(v,p_\sigma) = \sum_{\omega \in I} (Su_\sigma(\omega)-g_\omega)v(\omega) \quad &&\forall v \in W_0^{1,q}(\Omega). \label{eqn:system3b} 
\end{align}
\end{subequations}
Note that $p_\sigma$ is not a discrete function. The subscript $\sigma$ is to denote association with the discrete control $u_\sigma$.
\end{proposition}
\begin{proof}
$U_{ad,\sigma}$ is still a closed convex subset of $L^2(\Omega)$ so the proofs in Theorems \ref{thm:exist} and \ref{thm:optpoint} apply. Note that $p_\sigma= S^* \lambda_{u_\sigma}$ where $\lambda_{u_\sigma} \in \mathcal{M}(\Omega)$ is defined analogously to (\ref{eqn:lamu}) by
\begin{equation} \label{eqn:lamusig}
\lambda_{u_\sigma}(A) := \int_A (Su_\sigma-g_d) \rd \mu \quad \forall A \in \mathcal{B}.
\end{equation}
\end{proof}

Whereas (\ref{eqn:controlprob2point2}) minimises the discrete reduced objective functional over the continuous space, this problem minimises the continuous reduced objective functional over the discrete space. So the solution of (\ref{eqn:controlprob3}) is discrete, but the corresponding state is continuous, and this problem cannot be solved computationally.

The first term on the right hand side of (\ref{eqn:error}) can be thought of as the error from the discretisation of the control, as we are comparing the minimiser of the continuous objective functional over continuous and discrete controls. Similarly the second term on the right hand side of (\ref{eqn:error}) can be thought of as the error from the discretisation of the state, as we compare the minimiser of the continuous and discrete objective functionals, both over discrete controls. To prove Theorem~\ref{thm:main} it is sufficient to prove an error estimate for each term separately, which we do in Lemmas~\ref{lem:err1} and \ref{lem:err2}. Note that we have additional assumptions in Theorem~\ref{thm:main} because we need these in order to prove Lemma~\ref{lem:err1}. But first we will prove some a priori estimates for the solution of (\ref{eqn:controlprob3}).

\begin{lemma}
\label{lem:bounds}
Let $u_\sigma$ solve (\ref{eqn:controlprob3}) and $p_\sigma$ satisfy the optimality system (\ref{eqn:system3}). For all $q' \in (\frac{2n}{n+2}, \frac{n}{n-1} )$, 
\begin{equation}
\label{est:first}
\norm{u_\sigma}_{L^2(\Omega)} + \norm{Su_\sigma}_{L^2(\Omega)} + \norm{p_\sigma}_{W_0^{1,q'}(\Omega)} \leq C(q')
\end{equation}
with $C$ independent of $\sigma$.
Moreover, when $n=2$ there exists some $q>n$ such that
\begin{equation}
\label{est:second}
\norm{u_\sigma}_{L^q(\Omega)} + \norm{p_\sigma}_{L^q(\Omega)} \leq C
\end{equation}
with $C$ independent of $\sigma$.
\end{lemma}
\begin{proof}
Using (\ref{eqn:stabbb}), (\ref{eqn:lamusig}) and (\ref{eqn:cts}), for all $q' \in (\frac{2n}{n+2}, \frac{n}{n-1} )$ we have
\begin{align}
\norm{p_\sigma}_{W_0^{1,q'}(\Omega)} &\leq C(q') \norm{\lambda_{u_\sigma}}_{\mathcal{M}(\Omega)} \nonumber \\
&= C(q') \sum_{\omega \in I} \abs{Su_\sigma(\omega) - g_d} \nonumber \\
&\leq  C(q') (\norm{Su_\sigma}_\infty + \max_{\omega \in I} \abs{g_\omega} ) \nonumber \\
&\leq C(q') ( \norm{u_\sigma}_{L^2(\Omega}) + 1). \label{eqn:pbound}
\end{align}
Combining this with (\ref{eqn:cts}) again we get
\begin{equation}
\label{est:firsttt}
\norm{u_\sigma}_{L^2(\Omega)} + \norm{Su_\sigma}_{L^2(\Omega)} + \norm{p_\sigma}_{W_0^{1,q'}(\Omega)} \leq C(q')( \norm{u_\sigma}_{L^2(\Omega)} + 1).
\end{equation}
If $a,b \in \R$ then
\begin{align*}
\norm{u_\sigma}_{L^2(\Omega)} \leq \abs{\Omega}^{\frac{1}{2}} \max (\abs{a},\abs{b}).
\end{align*}
If $b=-a=\infty$ then $0 \in U_{ad,\sigma}$, so $\hat{J}(u_\sigma) \leq \hat{J}(0)$. Since $S 0 = 0$, this means
\begin{equation}
\label{eqn:fghj}
\frac{\nu}{2} \norm{u_\sigma}^2_{L^2(\Omega)} \leq \frac{1}{2} \sum_{\omega \in I} g_\omega^2.
\end{equation}
% \begin{align*}
% 0 = (p_\sigma,u_\sigma) + \nu \norm{u_\sigma}^2_{L^2(\Omega)} = \sum_{\omega \in I} (Su_\sigma(\omega) - g_\omega)Su_\sigma(\omega) + \nu \norm{u_\sigma}^2_{L^2(\Omega)}.
% \end{align*}
% So
% \begin{align*}
% \nu \norm{u_\sigma}^2_{L^2(\Omega)} + \sum_{\omega \in I} Su_\sigma(\omega)^2 = \sum_{\omega \in I} g_\omega Su_\sigma(\omega), 
% \end{align*}
% and hence
% \begin{align*}
% \nu \norm{u_\sigma}^2_{L^2(\Omega)} \leq \frac{1}{2} \sum_{\omega \in I} g_\omega^2.
% \end{align*}
So regardless of the assumptions on $a$ and $b$, we have $\norm{u_\sigma}_{L^2(\Omega)} \leq C$. Combining this with (\ref{est:firsttt}) gives the first bound (\ref{est:first}).

For the second bound we assume $n=2$. If $a,b \in \R$ then  we have
\[
\norm{u_\sigma}_{L^q(\Omega)} \leq \abs{\Omega}^{\frac{1}{q}} \max (\abs{a},\abs{b}).
\]
If $b=-a=\infty$ we can use the $L^q(\Omega)$ stability of the $L^2(\Omega)$ projection (see e.g.\ \cite{Crouzeix1987}) to get $\norm{u_\sigma}_{L^q(\Omega)} \leq \frac{1}{\nu}\norm{p_\sigma}_{L^q(\Omega)}$. So for all $q>2$,
\[
\norm{u_\sigma}_{L^q(\Omega)} + \norm{p_\sigma}_{L^q(\Omega)} \leq C( \norm{p_\sigma}_{L^q(\Omega)} + 1).
\]
We now need some $q>2$ such that $\norm{p_\sigma}_{L^q(\Omega)} \leq C$ independently of $\sigma$. By Sobolev embedding results, if $s> \frac{n}{2}$ then $W^{1,s}(\Omega) \hookrightarrow L^t(\Omega)$ for some $t>n$. In particular for $n=2$ we can take $s=\frac{4}{3}>\frac{n}{2}=1$, since $p_\sigma \in W_0^{1,s}(\Omega)$ for $s \in (\frac{2n}{n+2}, \frac{n}{n-1} )=(1,2)$. Then for some $q>2$,
\[
\norm{p_\sigma}_{L^q(\Omega)} \leq C \norm{p_\sigma}_{W_0^{1,\frac{4}{3}}(\Omega)} \leq C,
\]
where we have used (\ref{est:first}) for the final inequality. Note that for $n=3$ we would require $s>\frac{n}{2} = \frac{3}{2}$, but for such an $s$ we do not have $p_\sigma \in W_0^{1,s}(\Omega)$, which only holds when $s \in (\frac{2n}{n+2}, \frac{n}{n-1} )=(\frac{5}{4},\frac{3}{2})$.
\end{proof}

\begin{lemma}[Error from discretisation of the control]
\label{lem:err1}
Assume $n=2$. Let $u$ and $u_\sigma$ be solutions of (\ref{eqn:optprobpoint}) and (\ref{eqn:controlprob3}) respectively. Then
\[
\norm{u-u_\sigma}_{L^2(\Omega)} \leq C \sqrt{\sigma}
\]
with $C$ independent of $\sigma$ (and $h$).
\end{lemma}
\begin{proof}
Test with $v=u_\sigma$ in (\ref{eqn:systema}) to get
\[
(p + \nu u,u_\sigma-u) \geq 0.
\]
Test with $v_\sigma=\Pi_\sigma u$ in (\ref{eqn:system3a}) to get
\[
(p_\sigma + \nu u_\sigma,\Pi_\sigma u-u_\sigma) = (p_\sigma + \nu u_\sigma,\Pi_\sigma u - u) + (p_\sigma + \nu u_\sigma,u-u_\sigma) \geq 0.
\]
Adding these two inequalities and rearranging we get
\begin{equation}
\label{eqn:mainineq}
\nu \norm{u-u_\sigma}^2_{L^2(\Omega)} + (p_\sigma-p,u_\sigma-u) \leq (p_\sigma + \nu u_\sigma ,\Pi_\sigma u-u).
\end{equation}
Recall from the proof of Theorem~\ref{thm:optpoint} that $p = S^* \lambda_{u}$ with $\lambda_u$ defined by (\ref{eqn:lamu}). Similarly during the proof of Proposition~\ref{prop:tmp} we find that $p_\sigma = S^* \lambda_{u_\sigma}$ with $\lambda_{u_\sigma}$ defined by (\ref{eqn:lamusig}). So using this and Theorem~1.29 in \cite{Rudin} (see e.g.\ (\ref{eqn:measrel})) we get\begin{align*}
  (p_\sigma- p,u_\sigma-u) =& (S^*\lambda_{u_\sigma} - S^*\lambda_u, u_\sigma - u) = \langle \lambda_{u_\sigma} - \lambda_u, S(u_\sigma-u) \rangle_{\mathcal{M}(\Omega)}  \\=&  \int_\Omega (S(u-u_\sigma))^2 \rd \mu \geq 0.
\end{align*}
This means the  second term on the left hand side of (\ref{eqn:mainineq}) can be dropped. 

We now bound the right hand side of (\ref{eqn:mainineq}). By Lemma~\ref{lem:bounds}, for $n=2$ there exists some $q>n$ such that $\norm{u_\sigma}_{L^q(\Omega)}$ and $\norm{p_\sigma}_{L^q(\Omega)}$ are bounded independently of $\sigma$. So using H\"older's inequality with this $q$ we get
\begin{align*}
(p_\sigma + \nu u_\sigma, \Pi_\sigma u - u) &
\leq \norm{ p_\sigma + \nu u_\sigma }_{L^q(\Omega)} \norm{ \Pi_\sigma u - u }_{L^{q'}(\Omega)} \\ 
&\leq (\norm{ p_\sigma}_{L^q(\Omega)} + \nu \norm{u_\sigma }_{L^q(\Omega)}) \norm{ \Pi_\sigma u - u }_{L^{q'}(\Omega)} \\
&\leq C \norm{ \Pi_\sigma u - u }_{L^{q'}(\Omega)},
\end{align*}
with $C$ independent of $\sigma$. Now (\ref{eqn:approx}) gives
\[
\norm{ \Pi_\sigma u - u }_{L^{q'}(\Omega)} \leq C \sigma  \norm{ u }_{W^{1,q'}(\Omega)} \leq C \sigma,
\]
so we can deduce that
\[
(p_\sigma + \nu u_\sigma, \Pi_\sigma u - u) \leq C \sigma.
\]
Recall from Lemma~\ref{lem:bounds} that $\norm{u_\sigma}_{L^q(\Omega)}$ and $\norm{p_\sigma}_{L^q(\Omega)}$ are not bounded independently of $\sigma$ for $n=3$, so the above proof does not work in that case.
\end{proof}

\begin{lemma}[Error from discretisation of the state]
\label{lem:err2}
Assume $n=2$ or $3$. Let $u_\sigma$ and $u_{\sigma,h}$ be the solutions of (\ref{eqn:controlprob3}) and $(\mathrm{M}1_h)$ (see (\ref{eqn:controlprob2point})) respectively. Then
\[
\norm{u_\sigma - u_{\sigma,h}}_{L^2(\Omega)} \leq C h^{2-\frac{n}{2}}
\]
with $C$ independent of $\sigma$ and $h$.
\end{lemma}
\begin{proof}
Testing (\ref{eqn:system2apoint}) with $v_\sigma=u_{\sigma}$ gives
\[
(p_h + \nu u_{\sigma,h},u_\sigma-u_{\sigma,h}) \geq 0.
\]
Testing (\ref{eqn:system3a}) with $v_h=u_{\sigma,h}$ gives
\[
(p_\sigma + \nu u_{\sigma},u_{\sigma,h}-u_{\sigma}) \geq 0.
\]
Adding these two inequalities, using that $p_h=S_h^*\lambda_{h, u_{\sigma,h}}$ and $p_\sigma = S^* \lambda_{u_\sigma}$, and introducing $S_h^* \lambda_{h,u_{\sigma}}$ (see (\ref{eqn:meassy})) we get
%  is defined by
% \[
% \lambda_{h,u_{\sigma}}(A) = \int_A (S_h u_\sigma - g_d) \rd \mu \quad \forall A \in \mathcal{B}
% \]
\begin{align}
\nu \norm{u_\sigma - u_{\sigma,h}}^2_{L^2(\Omega)} \nonumber \leq& (p_h-p_{\sigma},u_\sigma-u_{\sigma,h}) \nonumber \\
=& (S_h^* \lambda_{u_{\sigma,h}}-S^* \lambda_{u_\sigma},u_\sigma-u_{\sigma,h}) \nonumber \\
\leq& (S_h^* \lambda_{h,u_{\sigma,h}}-S_h^*\lambda_{h,u_\sigma},u_\sigma-u_{\sigma,h}) \nonumber \\
&+ (S_h^*\lambda_{h,u_\sigma}-S^* \lambda_{u_\sigma},u_\sigma-u_{\sigma,h}). \label{eqn:splitted}
\end{align}
Note that
\begin{align*}
(S_h^* \lambda_{h,u_{\sigma,h}}-S_h^*\lambda_{h,u_\sigma},u_\sigma-u_{\sigma,h}) =& \langle \lambda_{h,u_{\sigma,h}}-\lambda_{h,u_\sigma}, S_h (u_\sigma - u_{\sigma,h}) \rangle_{\mathcal{M}(\Omega)} \\
 =& -\int_\Omega (S_h(u_{\sigma,h}-u_\sigma))^2 \rd \mu \leq 0.
\end{align*}
So the first term on the right hand side of (\ref{eqn:splitted}) can be dropped. Also note that
\begin{align*}
(S_h^* \lambda_{h,u_\sigma} - S^* \lambda_{u_\sigma} , u_\sigma-u_{\sigma,h})=& (S_h^*\lambda_{h,u_\sigma}-S^*\lambda_{h,u_\sigma},u_\sigma-u_{\sigma,h}) +(S^*\lambda_{h,u_\sigma}-S^*\lambda_{u_\sigma}, u_\sigma-u_{\sigma,h}),
%  \\
% &\leq \left( \norm{S_h^*\lambda_h(u_\sigma)-S_h^*\lambda(u_\sigma) }_{L^2(\Omega)}+ \norm{S_h^*\lambda(u_\sigma)-S^*\lambda(u_\sigma)}_{L^2(\Omega)} \right) \norm{ u_\sigma-u_{\sigma,h} }_{L^2(\Omega)} \\
% &\leq C h^{2-\frac{n}{2}} \norm{\lambda_h(u_\sigma)-\lambda(u_\sigma)}_{\mathcal{M}(\Omega)} \\
% &\leq C \int_\Omega \abs{S_h u_\sigma - S u_\sigma } \\
% &\leq C \norm{S_h u_\sigma - S u_\sigma}_\infty
\end{align*}
and we can bound both terms on the right hand side of this. Using (\ref{eqn:err2}) and
\begin{equation}
\label{eqn:mbound}
\norm{\lambda_{h,u_\sigma}}_{\mathcal{M}(\Omega)} = \sum_{\omega \in I} \abs{S_hu_\sigma(\omega)-g_\omega} \leq C \norm{S_hu_\sigma}_\infty + \max_{\omega \in I} \abs{g_\omega} \leq C( \norm{u_\sigma}_{L^2(\Omega)} +1) \leq C,
\end{equation}
we get
\begin{align}
(S_h^*\lambda_{h,u_\sigma}-S^*\lambda_{h,u_\sigma},u_\sigma-u_{\sigma,h}) &\leq C \norm{S_h^*\lambda_{h,u_\sigma}-S^*\lambda_{h,u_\sigma}}_{L^2(\Omega)} \norm{ u_\sigma-u_{\sigma,h} }_{L^2(\Omega)} \nonumber \\
&\leq C h^{2-\frac{n}{2}} \norm{ \lambda_{h,u_\sigma} }_{\mathcal{M}(\Omega)} \norm{ u_\sigma-u_{\sigma,h} }_{L^2(\Omega)} \nonumber \\
&\leq C h^{2-\frac{n}{2}} \norm{ u_\sigma-u_{\sigma,h} }_{L^2(\Omega)}  \label{eqn:split1}
\end{align}
with $C$ independent of $\sigma$ and $h$. By (\ref{eqn:errinf}) we have 
\begin{align}
(S^*\lambda_{h,u_\sigma}-S^*\lambda_{u_\sigma}, u_\sigma-u_{\sigma,h}) &\leq C \norm{S^*\lambda_{h,u_\sigma}-S^*\lambda_{u_\sigma}}_{L^2(\Omega)} \norm{ u_\sigma-u_{\sigma,h} }_{L^2(\Omega)} \nonumber \\
&\leq C \norm{ \lambda_{h,u_\sigma} - \lambda_{u_\sigma} }_{\mathcal{M}(\Omega)} \norm{ u_\sigma-u_{\sigma,h} }_{L^2(\Omega)} \nonumber \\
&= C \left ( \sum_{\omega \in I} \abs{ S_h u_\sigma(\omega) -S u_\sigma(\omega) } \right) \norm{ u_\sigma-u_{\sigma,h} }_{L^2(\Omega)} \nonumber \\
&\leq C \norm{ S_h u_\sigma-S u_\sigma }_{L^\infty(\Omega)} \norm{ u_\sigma-u_{\sigma,h} }_{L^2(\Omega)} \nonumber \\
&\leq C h^{2-\frac{n}{2}} \norm{u_\sigma}_{L^2(\Omega)} \norm{ u_\sigma-u_{\sigma,h} }_{L^2(\Omega)} \nonumber \\
&\leq C h^{2-\frac{n}{2}} \norm{ u_\sigma-u_{\sigma,h} }_{L^2(\Omega)} \label{eqn:split2} %could change proof to make it more similar to first lemma
\end{align}
with $C$ independent of $\sigma$ and $h$. So
\[
(S^*\lambda_{h,u_\sigma}-S^*\lambda_{u_\sigma}, u_\sigma-u_{\sigma,h}) \leq C h^{2-\frac{n}{2}}  \norm{u_\sigma-u_{\sigma,h}}_{L^2(\Omega)},
\]
and using this in (\ref{eqn:splitted}) completes the proof. % gives 
% \[
% \norm{u_\sigma - u_{\sigma,h}}_{L^2(\Omega)} \leq C h^{2-\frac{n}{2}},
% \]
% completing the proof.
% The observation that $\norm{u_\sigma}_{L^2(\Omega)} \leq C \norm{u}_{L^2(\Omega)}$ for sufficiently small $\sigma$ (see the end of Lemma~\ref{lem:err1}) now completes the result.
\end{proof}

Combining Lemmas \ref{lem:err1} and \ref{lem:err2} gives Theorem~\ref{thm:main}. A consequence of the theorem is that in 2 dimensions by taking $\sigma=h^2$ we can get an $O(h)$ error estimate for the explicitly discretised problem $(\mathrm{M}1_h)$. In this case the state is on a triangulation of size $O(h)$ and the control is on a triangulation of size $O(h^2)$ (i.e.\ the control space on a finer triangulation than the state). Even though a finer triangulation is involved, the PDEs are posed on the state space to it is reasonable to think of this error estimate as $O(h)$. 

Note that Theorem~\ref{thm:main22} can be proved using the same sequence of calculations and bounds as Lemma~\ref{lem:err2}. To see this observe that if we replace $U_{ad,\sigma}$ by $U_{ad}$ in both (\ref{eqn:controlprob3}) and $(\mathrm{M}1_h)$, then $u_\sigma$ solves the continuous problem (\ref{eqn:controlprob}) and $u_{\sigma,h}$ solves the implicitly discretised problem $(\mathrm{M}2_h)$. %The only properties of $u_\sigma$ and $u_{\sigma,h}$ that we use in the lemma is that $\norm{u_\sigma}_{L^2(\Omega)}$ is bounded independently of $\sigma$, which is still true if $u_\sigma$ instead solves (\ref{eqn:controlprob}) .

\begin{remark} \label{rem:obs}
As we noted in Remark~\ref{rem:discrep}, sometimes $(\mathrm{M}1_h)$ is equivalent to $(\mathrm{M}2_h)$. In these cases (e.g.\ when there are no active control constraints and $V_h \subset U_{\sigma}$) Theorem~\ref{thm:main22} gives error estimates for $(\mathrm{M}1_h)$. This observation proves Corollary~\ref{cor:thm12}. In particular it gives an estimate for $(\mathrm{M}1_h)$ when $n=3$ without control constraints, which Theorem~\ref{thm:main} does not provide.
\end{remark}

\begin{remark}
Using this approach to the numerical analysis, the error estimate analogous to Theorem~\ref{thm:main} for a control problem with an $L^2(\Omega)$ fidelity term (instead of one containing point evaluations) is $O(\sigma + h^2)$ (see \cite{Casas2003}).
\end{remark}

% \begin{remark}
% Whereas the proof of Lemma~\ref{lem:err1} used an interpolation operator that required us to take the control in a space of piecewise linear finite elements, note that Lemma~\ref{lem:err2} does not. In particular, Theorem~\ref{thm:main2} still holds if we take the control space to be a space of piecewise constant finite elements i.e.\ take 
% \[
% U_\sigma := \{u_\sigma \in L^2(\Omega) : u_\sigma|_T \in P_0(T) \text{ for all } T \in T^\sigma \},
% \]
% where $P_0(T)$ denotes the set of functions that are constant on $T$.
% \end{remark}

\subsection{Approach 2}
\label{sec:energy}

This error analysis is based on the technique used in \cite{Deckelnick2007} and \cite{Leykekhman2013}. The approach applies to the implicit discretisation $(\mathrm{M}2_h)$ (see (\ref{eqn:controlprob2point2})), and therefore also to the explicit discretisation $(\mathrm{M}1_h)$ (see (\ref{eqn:controlprob2point})) when these discrete problems are equivalent (see Remark~\ref{rem:discrep}). However it does not apply to $(\mathrm{M}1_h)$ in general.
%here are no active control constraints. So unlike Corollary~\ref{cor:thm12}, we have the flexibility of using different discrete spaces for the control and state in $(\mathrm{M}1_h)$. 
%mistake, U_\sigma must be quite good approx?
%piecewise constant discrete space

The key ingredient of Approach~2 is bounding the difference between the continuous reduced objective functional applied to the discrete and continuous optimal controls, and similarly for the discrete reduced objective functional. Instead of needing error estimates for the control-to-state operator and its adjoint, which were required to prove Theorem~\ref{thm:main22}, this approach only uses the strong supremum norm error estimate (\ref{eqn:precise}). It also does not require the manipulation of measures. As a result this approach is mathematically simpler than Approach~1, but it still allows us to prove the same error estimate as in Theorem~\ref{thm:main22} (modulo $\eps$).

%However it has a simple proof, so we introduce it ready for the next chapter, where it is used for optimal control involving surfaces of codimension 1.

% \begin{lemma}
% \label{lem:infest}
% Let $u$ and $u_h$ be solutions of (\ref{eqn:controlprob}) and (\ref{eqn:controlprob3}). Then
% \begin{align*}
% \norm{y(u)-y_h(u)}_\infty, \norm{y(u_h)-y_h(u_h)}_\infty  &\leq Ch^{4-n} \abs{\log h}^{7-2n} % \max_{\omega \in I} \abs{g_\omega}, \\
% \end{align*}
% with $C$ independent of $h$.
% \end{lemma}
% \begin{proof}
% Since $u$ is a solution of (\ref{eqn:controlprob}), $u \in W_0^{1,q'}(\Omega)$ for $q' \in (1,\frac{n}{n-1})$. Also
% \[
% \norm{u}_{W_0^{1,q'}(\Omega)} = \frac{1}{\nu} \norm{p}_{W_0^{1,q'}(\Omega)} \leq  C, %C \max_{\omega \in I} \abs{y}_\omega ,
% \]
% where $C$ is independent of $q'$. Using with (\ref{eqn:precise}) we get
% \[
% \norm{y(u)-y(u_h)}_\infty \leq C \left(\frac{nq'}{n-q'} \right)^2 h^{3-\frac{n}{q'}}.
% \]
% If we follow the approach in Corollary 1 of \cite{Leykekhman2013} and take
% \[
% q'=\frac{n+\eps}{n-1+\eps},
% \]
% for some $\eps>0$ then we can calculate that
% \[
% \left(\frac{nq'}{n-q'} \right)^2 h^{3-\frac{n}{q'}} < \frac{h^{4-n-\frac{\eps}{n}}}{\eps^{7-2n}}.
% \]
% Now setting $\eps= \abs{\log h }^{-1}$ gives the first inequality. 

% The proof of the second inequality follows in the same way.
% \end{proof}

\begin{theorem}
\label{thm:main2}
Let $u$ be a solution of (\ref{eqn:controlprob}) and $u_h$ be a solution of $(\mathrm{M}2_h)$ (see (\ref{eqn:controlprob2point2})). Then for any $\eps>0$,
\[
\norm{u-u_h}_{L^2(\Omega)} \leq C(\eps) h^{2-\frac{n}{2}-\eps} %\abs{ \log h }^{\frac{7}{2}-n} %notation for discrete solution
\]
with $C$ independent of $h$. %Compare this to Theorem~\ref{thm:main22}, which has the same assumptions.
\end{theorem}
\begin{proof}
First observe that
\begin{align} \label{eqn:jcomp}
\hat{J}(u_h)-\hat{J}(u) =& \frac{1}{2} \sum_{\omega \in I} (Su_h-Su)(\omega)^2 + \frac{\nu}{2} \norm{ u_h-u }^2_{L^2(\Omega)} \nonumber \\ &+ \sum_{\omega \in I} (Su_h-Su)(Su-g_\omega)(\omega) + \nu (u,u_h-u) \nonumber \\
\geq& \frac{1}{2} \sum_{\omega \in I} (Su_h-Su)(\omega)^2 + \frac{\nu}{2} \norm{ u_h-u }^2_{L^2(\Omega)},
\end{align}
since the optimality conditions imply that
\[
\sum_{\omega \in I} (Su_h-Su)(Su-g_\omega)(\omega) = a(Su_h-Su,p)=(u_h-u,p)\geq-\nu (u_h-u,u).
\]
Similarly
\begin{align} \label{eqn:jhcomp}
\hat{J}_h(u)-\hat{J}_h(u_h) \geq \frac{1}{2} \sum_{\omega \in I} (S_h u_h-S_h u)(\omega)^2 + \frac{\nu}{2} \norm{ u_h-u }^2_{L^2(\Omega)}.
\end{align}
Note that the final inequality in this calculation holds for $(\mathrm{M}2_h)$ but not for $(\mathrm{M}1_h)$ without additional assumptions.

So combining (\ref{eqn:jcomp}) and (\ref{eqn:jhcomp}) we get
\begin{align} 
\nu \norm{u-u_h}^2_{L^2(\Omega)} &\leq \hat{J}(u_h)-\hat{J}(u) + \hat{J}_h(u)-\hat{J}_h(u_h) \nonumber \\
&\leq \abs{\hat{J}(u) - \hat{J}_h(u)} + \abs{\hat{J}(u_h) -\hat{J}_h(u_h) }. \label{eqn:engbnd}
\end{align}
We can bound each of the terms on the right hand side of this inequality. Note that
\begin{align}
\abs{\hat{J}(u)-\hat{J}_h(u)} &= \abs{\frac{1}{2} \sum_{\omega \in I} (Su(\omega)-g_\omega)^2 - \frac{1}{2} \sum_{\omega \in I} (S_h u(\omega)-g_\omega)^2} \nonumber \\
&=\abs{\frac{1}{2} \sum_{\omega \in I} (Su-S_hu)(Su-g_\omega+S_hu-g_\omega)(\omega)} \nonumber \\
&\leq C \norm{ Su-S_hu }_\infty (\norm{Su}_\infty + \norm{S_hu}_\infty + \max_{\omega \in I} \abs{g_\omega}) \nonumber \\
&\leq C \norm{ Su-S_hu }_\infty ( \norm{u}_{L^2(\Omega)} + 1) . \nonumber %\max_{\omega \in I} \abs{g_\omega}. \nonumber
\end{align}
So (\ref{eqn:precise}) gives that for all $q' \in (\frac{2n}{n+2}, \frac{n}{n-1} )$,
\begin{align}
\Big |\hat{J}(u)-\hat{J}_h(u) \Big | 
&\leq C(q')h^{3-\frac{n}{q'}} \norm{u}_{W^{1,q'}(\Omega)}( \norm{u}_{L^2(\Omega)} + 1) \nonumber \\ 
&\leq C(q')h^{3-\frac{n}{q'}}. \label{eqn:ucomp} % \abs{ \log h }^{7-2n}. %  \max_{\omega \in I} \abs{g_\omega}.
\end{align}
% elaborate on bound for last inequality
In the same way we get that for all $q' \in (\frac{2n}{n+2}, \frac{n}{n-1} )$,
\begin{align} 
\abs{\hat{J}(u_h)-\hat{J}_h(u_h)} &\leq C \norm{ Su_h-S_hu_h }_\infty (\norm{Su_h}_\infty + \norm{S_hu_h}_\infty + \max_{\omega \in I} \abs{g_\omega}) \nonumber \\
&\leq C(q') h^{3-\frac{n}{q'}} \norm{u_h}_{W^{1,q'}(\Omega)}( \norm{u_h}_{L^2(\Omega)} + 1). \label{eqn:uhcomp} % \abs{ \log h }^{7-2n}. % \max_{\omega \in I} \abs{g_\omega}.
\end{align}
Since $u_h = \mathbb{P}_{[a,b]} (-\frac{1}{\nu}p_h)$ we have $\norm{u_h}_{W^{1,q'}(\Omega)} \leq C \norm{p_h}_{W_0^{1,q'}(\Omega)}$, and the same calculation as in the beginning of Lemma~\ref{lem:bounds} gives that
\[
\norm{p_h}_{W_0^{1,q'}(\Omega)} \leq C(q')
\]
independently of $h$. Combining this, (\ref{eqn:engbnd}), (\ref{eqn:ucomp}) and (\ref{eqn:uhcomp}) gives
\[
\norm{u-u_h}^2_{L^2(\Omega)} \leq C(q') h^{3-\frac{n}{q'}}.
\]
Now for any $\eps>0$ we can set 
\[
 q'=\frac{n}{n-1+2\eps},
\]
which completes the proof of the theorem.
\end{proof}

\begin{remark}
In this proof we used the strong supremum norm estimate (\ref{eqn:precise}) rather than (\ref{eqn:errinf}). This cannot be used to improve the estimates from Approach~1 in Section~\ref{sec:semidisc}; supremum norm estimates are not used in Lemma~\ref{lem:err1}, and in Lemma~\ref{lem:err2} we can improve the bound in (\ref{eqn:split2}) but the error would still be dominated by the $h^{{2-\frac{n}{2}} }$ term in (\ref{eqn:split1}).
\end{remark}

\subsection{Forcing term}
\label{sec:force}

We did not include a forcing term in our write up in order to simplify the presentation. However all the results we have proved still hold if we include a forcing term $f$ in the state equation with the regularity $f \in W_0^{1,q'}(\Omega)$ for all $q' \in (\frac{2n}{n+2}, \frac{n}{n-1} )$. In particular, if we replace (\ref{eqn:state}) by
\begin{equation}
a(y,v) = (\eta+f,v) \quad \forall v \in H_0^1(\Omega), \label{eqn:statenew}
\end{equation}
and consider a control problem of the form
\begin{equation*}
\begin{aligned}
\min \quad &J(y,\eta) := \frac{1}{2}\sum_{\omega \in I} (y(\omega)-g_\omega)^2 + \frac{\nu}{2} \norm{\eta}^2_{L^2(\Omega)} \\
\text{over } \quad &C_0(\Omega) \times L^2(\Omega) \\
\text{s.t. } \quad &(\ref{eqn:statenew}) \text{ holds} \\
\text{and } \quad &\eta \in U_{ad} :=\{\eta \in L^2(\Omega) : a \leq \eta \leq b \text{ a.e.\ in } \Omega \}
\end{aligned}
\end{equation*}
with all other assumptions the same as in (\ref{eqn:controlprob}). This problem has the reduced form
\begin{equation}
\label{eqn:probwithforce}
\begin{aligned}
\min \quad &\hat{J}(\eta) := \frac{1}{2} \sum_{\omega \in I} (S(\eta+f)(\omega)-g_\omega)^2 + \frac{\nu}{2} \norm{\eta}^2_{L^2(\Omega)} \\
\text{over } \quad &\eta \in U_{ad},
\end{aligned}
\end{equation}
where $S$ is as defined previously. For this problem we can construct non-trivial examples with explicitly known solutions (see Section~\ref{sec:exact}), which we cannot do for the problem without a forcing term. So after extending our theory to include a forcing term we are able to perform some numerical experiments to verify that our error estimates are observed in practice. 

The forcing term means that the mapping from $\eta$ to $y$ defined by the state equation is no longer linear but instead affine. This difference can be handled with only minor modifications to our problem formulations and proofs, which we now mention: The optimal control problem with forcing still has a unique solution (see e.g.\ Theorem~1.45 in \cite{Hinze2009}). Corollary~1.3 in \cite{Hinze2009} gives that $u$ solves (\ref{eqn:probwithforce}) if and only if $u$ solves (\ref{eqn:system}) with $Su$ replaced by $S(u+f)$ i.e.\ for all $q' \in (\frac{2n}{n+2}, \frac{n}{n-1} )$ there exist $p \in W_0^{1,q'}(\Omega)$ such that
\begin{align*}
&u \in U_{ad}, \quad (p - \nu u, v - u) \geq 0  &&\forall v \in U_{ad}, \\
%&a(y,v)=(u,v) &&\forall v \in H_0^1(\Omega) \\
&a(v,p)= \sum_{\omega \in I} (S(u + f)(\omega)-g_\omega)v(\omega) &&\forall v \in W_0^{1,q}(\Omega).
\end{align*}
The same reasoning applies to the discrete problems and their optimality conditions with the obvious modifications. In particular the optimality conditions for the discrete problem $(\mathrm{M}2_h)$  (see (\ref{eqn:controlprob2point2})) with the inclusion of the forcing term are: There exists a $p_h \in V_h$ such that
\begin{subequations}
\label{eqn:optforce}
\begin{align}
&u_h \in U_{ad}, \quad (p_h + \nu u_h,v-u_h) \geq 0 && \forall v \in U_{ad}, \label{eqn:optforcea}  \\
%&a(y_h,v_h) = (u_{\sigma,h},v_h) \quad \forall v_h \in V_h \\
&a(v_h,p_h) = \sum_{\omega \in I} (S_h(u_h+f)(\omega)-g_\omega)v_h(\omega) && \forall v_h \in V_h. \label{eqn:optforceb} 
\end{align}
\end{subequations}
Theorems~\ref{thm:main}, \ref{thm:main2} and \ref{thm:main22} still hold with same methods of proof; the $f$ term slightly alters the calculations but does not cause problems, since it follows immediately from the supremum norm error estimate (\ref{eqn:precise}) that for $\eta \in W_0^{1,q'}(\Omega)$ with $q' \in (\frac{2n}{n+2}, \frac{n}{n-1} )$,
\[
\norm{(S-S_h)(\eta+f)}_\infty \leq C(q') h^{3-\frac{n}{q'}} \norm{\eta+f}_{W_0^{1,q'}(\Omega)}.
\]

%The numerical approach which we now define will includes a forcing term.

% \subsection{Variational discretisation}

% We can also prove error estimates for problem (\ref{eqn:controlprob2point2}), which uses the variational discretisation idea of \cite{Hinze2005}. In fact we note that in the proof of Lemma \ref{lem:err2}, $U_{ad,\sigma}$ can be replaced by $U_{ad}$, proving the following result, which is valid with and without control constraints..

% \begin{theorem}
% Let $u$ and $u_h$ be the solutions of (\ref{eqn:controlprob2point}) and (\ref{eqn:controlprob2point2}) respectively. Then
% \[
% \norm{u-u_h}_{L^2(\Omega)} \leq C h^{\frac{n}{n-1}}
% \]
% with $C$ independent of $h$.
% \end{theorem}

% Note that although proving error estimates for (\ref{eqn:system22}) is fairly straightforward, more work is needed to prove that it can be solved numerically, and the implementation is more complicated. 

% % The proved error estimates are the same as with the previous two approaches, but typically this approach leads to a lower absolute error as the sets where the control constraints are active can be better approximated.

\section{Numerical results}
\label{sec:pointnumerics}

In this section we develop a numerical method for solving $(\mathrm{M2}_h)$ with modification to include a forcing term (see (\ref{eqn:optforce})) and show that the a priori $L^2(\Omega)$ error estimates that we proved for this discrete problem are numerically realised. In order to do this we solve simple examples of the optimal control problems with explicitly known solutions. We also include some simulations for more interesting problems for which the exact solutions are not known.

\subsection{Numerical method}
\label{sec:nummeth}

We only develop a numerical method for solving $(\mathrm{M}2_h)$ because we are able to prove better error estimates for this discrete problem. In particular, we do not have an error estimate for  $(\mathrm{M}1_h)$ when $n=3$ with control constraints. Perhaps such an estimate could be proved in other ways, but we cannot easily experimentally investigate if it holds either; we only have explicit solutions (which allow us to reliably test error estimates) when there are no active control constraints. We will now describe the numerical method. % we use for solving this discrete problem (with a forcing term -- see Section~\ref{sec:force}).

% Recall from Remark \ref{rem:discrep} that we formulate a numerical method for soliving the discretisation (\ref{eqn:system2apoint2}) based on the variational discretisation concept of \cite{Hinze2005}. However in the case of no active control constraints (and in particular when there are no control constraints), this discretisation is the same as (\ref{eqn:system2apoint}) with the control in the same discrete space as the state. In our examples with explicitly known solutions there are no control constraints, so our numerical method is suitable for testing the optimality of our error estimates for these examoples. There is discrepency between our numerical method and our error estimates when control constraints are active, but we can only heuristically measure rates of convergence in this case anyway.

If $u_h$ solves (\ref{eqn:optforce}), then by substituting $u_h=\mathbb{P}_{[a,b]}(-\frac{1}{\nu}p_h)$ we get that the state $y_h:=S_h u_h \in V_h$ and the adjoint variable $p_h \in V_h$ solve
\begin{equation}
\label{eqn:newtonfunpoint}
\left( \begin{array}{c} 
a(y_h, v_h ) - (-\frac{1}{\nu} p_h + (a + \frac{1}{\nu} p_h )^+ - (-\frac{1}{\nu} p_h
-b )^+ - f, v_h ) \\
 a(w_h,p_h) - \sum_{\omega \in I} (y_h(\omega)-g_\omega) w_h (\omega) \end{array} \right) = 0
\end{equation}
% \begin{equation}
% %\label{eqn:newtonfunpoint}
% \left( \begin{array}{c} 
% a(y_h, v_h ) - (u_h + f, v_h ) \\
%  a(v_h,p_h) + \sum_{\omega \in I} (y_h(\omega)-g_\omega) v_h (\omega)\\
% u_h + (\frac{1}{\nu} p_h - (a + \frac{1}{\nu} p_h )^+ + (-\frac{1}{\nu} p_h-b )^+,v_h)
%  \end{array} \right) = 0 \quad \forall v_h \in V_h.
% \end{equation}
for all $v_h,w_h \in V_h$. Here $v^+$ denotes the nonnegative part of $v$ i.e.\ $\max(0,v)$. Once this problem has been solved, the $u_h$ solving (\ref{eqn:system2apoint2}) can easily be determined from $p_h$ by setting $u_h=\mathbb{P}_{[a,b]}\big (-\frac{1}{\nu} p_h \big)$. We will now describe a numerical method for solving (\ref{eqn:newtonfunpoint}) with and without control constraints.

% In our implementation we numerical integration and so only solve a close approximation of (\ref{eqn:newtonfunpoint}). However if desired we could use our numerical method to solve this problem exactly without using numerical integration. We will discuss this further later in the section.

\subsubsection{No control constraints}

In the case of no control constraints ($b=-a=\infty$) the nonlinear $\max(0,\cdot)$ terms drop out, leaving a linear problem. Let $y_h = \sum_{z \in \mathcal{N}} y_z \varphi_z$ and $p_h = \sum_{z \in \mathcal{N}} p_z \varphi_z$, where $\varphi_z$ are the usual nodal basis functions of $V_h$ (defined by $\varphi_z(\bar{z}) = \delta_{z\bar{z}}$ for $\bar{z} \in \mathcal{N}$, where $\delta_{z \bar{z}}$ denotes the Kronecker delta and $\mathcal{N}$ is the set of interior vertices of the triangulation), and $y_z$ and $p_z$ are the coefficients corresponding to the basis functions. As we have no control constraints, testing (\ref{eqn:newtonfunpoint}) with $v_h = \varphi_z$ and $w_h = \varphi_{\bar{z}}$ for all  $z, \bar{z} \in \mathcal{N}$ leads to a system of linear equations of real variables. In particular, let $\bar{y}$ and $\bar{p}$ be vectors of coefficients defined by $\bar{y}_z = y_z$ and $\bar{p}_z = p_z$ for $z \in \mathcal{N}$ i.e.\ use the set of interior vertices as an index. Then we can solve (\ref{eqn:newtonfunpoint}) by solving the system of linear equations
\[
\begin{pmatrix}
A & \frac{1}{\nu}M \\
-\sum_{\omega \in I}M_\omega & A \\
\end{pmatrix}
\begin{pmatrix}
\bar{y} \\ \bar{p}
\end{pmatrix}
=
\begin{pmatrix}
\bar{F} \\ -\sum_{\omega \in I} \bar{G}_\omega
\end{pmatrix},
\]
where matrices $A$, $M$ and $M_\omega$ and vectors $\bar{F}$ and $\bar{G}_\omega$ are defined by
\begin{align*}
&A_{z \bar{z}} = a(\varphi_z, \varphi_{\bar{z}}), \quad M_{z \bar{z}} = (\varphi_z, \varphi_{\bar{z}}), \quad (M_\omega)_{z \bar{z}} = \varphi_z(\omega)\varphi_{\bar{z}}(\omega) \quad &&\forall z, \bar{z} \in \mathcal{N}, \\
&\bar{F}_z = (f,\varphi_z), \quad (\bar{G}_\omega)_z = g_\omega \varphi_z(\omega) \quad &&\forall z \in \mathcal{N}.
\end{align*}
As the basis functions $\varphi_z$ are piecewise linear with small support, the integrals that form the elements of the matrices and vectors are straightforward to compute, assuming $A$ and $f$ have a simple form (or else numerical integration of some terms may be required, which we discuss later). The matrix in this system of equations is sparse and so the system can be solved efficiently.

\subsubsection{Control constraints}

In the case of control constraints the nonlinear $\max(0,\cdot)$ terms mean that we can no longer use the above approach to construct a linear system of equations of real variables. Instead we will solve the problem iteratively using a Newton-type method. Let $F_h :V_h \times V_h \to V_h^* \times V_h^*$ with $F_h(y_h,p_h)(w_h, v_h)$ defined by the left hand side of (\ref{eqn:newtonfunpoint}). Then it can be written as
\begin{equation}
\label{eqn:newtf}
F_h(y_h,p_h) = 0 \quad \text{ in } V_h^* \times V_h^*.
\end{equation}
The $\max(0,\cdot)$ terms mean that $F_h$ is not Fr\'{e}chet differentiable. However we can apply a generalised Newton method called the semismooth Newton method (see e.g.\ \cite{Ulbrich2002} and \cite{Hintermuller2009}). This amounts to applying the Newton method in the usual way but taking the derivative of $\max(0,x)$ to be
\[
{\max}'(0,x) = 
\begin{cases}
1 &x \geq 0, \\
0 &x < 0. \\
\end{cases}
\]
So we take an initial guess $y_h^0, p_h^0$ then apply until the convergence the semismooth Newton iteration
\[
\left (\begin{array}{c}
y_h^{n+1}\\ p_h^{n+1}
\end{array} \right )
= \left ( \begin{array}{c}
y_h^{n}\\ y_h^{n}
\end{array} \right )
+ \left ( \begin{array}{c}
\delta y_h\\ \delta p_h
\end{array} \right ),
\]
where $\delta y_h$, $\delta p_h$ solve
\begin{equation}
\label{eqn:sysss}
\begin{aligned}
&F_h' (y_h^n,p_h^n)(\delta y_h, \delta p_h) \\
&=\left( \begin{array}{c} 
a(\delta y_h, \cdot ) - \frac{1}{\nu} \Big( \big (-1 + \max'(0,a + \frac{1}{\nu} p^n_h ) + \max'(0,-\frac{1}{\nu} p^n_h-b )  \big )  \delta p_h, \cdot \Big ) \\
 a(\cdot,\delta p_h) - \sum_{\omega \in I} \delta y_h(\omega)(\cdot)(\omega) \end{array} \right) \\
&= -F_h(y^n_h,p^n_h).
\end{aligned}
\end{equation} 
Note that if we have no control constraints the first Newton iteration is equivalent to solving (\ref{eqn:newtonfunpoint}).

% Note that a numerical method for solving this problem exactly is implementable (see \cite{Hinze2005}) but complicated, as we need integrate the $\max(0,\cdot)$ terms and $f$ terms. In our implementation we do this integration numerically, so we will only solve a close approximation of (\ref{eqn:newtonfunpoint}). As the algorithm we describe could be implemented without numerical integration, we will leave discussion of this until towards the end of the section.

As before we can represent $\delta y_h$ and $\delta p_h$ as a sum of basis functions weighted by coefficients, and testing (\ref{eqn:sysss}) with the basis functions allows us to construct the following system of linear equations of real variables:
\[
\begin{pmatrix}
A & \frac{1}{\nu}M_c \\
-\sum_{\omega \in I}M_\omega & A \\
\end{pmatrix}
\begin{pmatrix}
\delta \bar{y} \\ \delta \bar{p}
\end{pmatrix}
=
\begin{pmatrix}
\bar{R}_1 \\ \bar{R}_2
\end{pmatrix},
\]
where
\begin{align*}
(M_c)_{z \bar{z}} &:= (c(x) \varphi_z, \varphi_{\bar{z}}) \quad && \forall z, \bar{z} \in \mathcal{N}, \\
& \text{with } c(x) := 1 - {\max}'(0,a + \frac{1}{\nu} p^n_h(x) ) - {\max}'(0,-\frac{1}{\nu} p^n_h(x)-b ), && \\ 
(\bar{R}_{1})_z &:= -(F_h(y^n_h,p^n_h)_1, \varphi_z), \quad (\bar{R}_{2})_z := -(F_h(y^n_h,p^n_h)_2, \varphi_z) \quad && \forall z \in \mathcal{N}.
\end{align*}
Note that since $p^n_h(x)$ is piecewise linear, the integrals of functions such as ${\max}'(0,a + \frac{1}{\nu} p^n_h ) \varphi_z \varphi_{\bar{z}}$ can be computed exactly. In practice we instead approximate this using a numerical quadrature i.e.\ instead of $ (c(x) \varphi_z, \varphi_{\bar{z}})$ we will compute $Q(c(x) \varphi_z \varphi_{\bar{z}})$, where
\[
Q(\eta):=\sum_{T \in T_h} Q_T(\eta), \quad Q_T(\eta) := \sum_{q=1}^K \hat{w}_q \abs{DF_T(\hat{x}_q)} \eta(F_T(\hat{x}_q) ).
\]
Here $\{(\hat{w}_q,\hat{x}_q)\}_{q=1}^K$ is a collection of $K$ pairs of weights and points on a reference element $\hat{T}$ and $F_T$ is the reference mapping between $\hat{T}$ and $T$. We will use a Gaussian quadrature of high order (large $K$), so $Q(\eta) \approx \int_\Omega \eta(x) \rd x$. We will also use this quadrature rule to approximate $f$ as it may have a form that makes it complicated to integrate by hand. The moderately large error from our discretisation should dominate the smaller error from Gaussian quadrature (as it has good approximation properties), so we do not expect using quadrature to affect the $L^2(\Omega)$ error we observe in practice. Note that using quadrature means that we are not solving  (\ref{eqn:newtonfunpoint}) but rather a close approximation. Although using quadrature is not strictly necessary, %s exact computation of $(b(x) \varphi_z, \varphi_{\bar{z}})$ for piecewise linear finite elements is implementable (see \cite{Hinze2005}). 
the implementation without would require us to do additional calculations by hand, particularly in 3 dimensions. In comparison, there is built in support for numerical quadrature in many finite element software packages. %, and we believe the error caused by this approximation to be negligible.

Define the product space norm for $(z_1,z_2) \in Z \times Z$, where $Z$ is a normed vector space, by  $\norm{(z_1,z_2)}_{Z} = \sqrt{\norm{z_1}^2_Z + \norm{z_2}^2_Z}$. For $z \in H^{-1}(\Omega)$ let $w \in H_0^1(\Omega)$ be defined by
\[
(\nabla w,\nabla v)=\langle z,v \rangle_{H^{-1}(\Omega)} \quad \forall v \in H_0^1(\Omega).
\]
Then
\begin{align*}
\norm{z}_{H^{-1}(\Omega)} &= \sup_{v \in H_0^1(\Omega)} \frac{\langle z, v \rangle_{H^{-1}(\Omega)}}{\norm{v}_{H_0^1(\Omega)}} \\
&= \sup_{v \in H_0^1(\Omega)} \frac{ (\nabla w, \nabla v)} {\norm{v}_{H_0^1(\Omega)}} \\
&= \norm{w}_{H_0^1(\Omega)}.
\end{align*}
This motivates us to iterate the Newton method until the stopping criterion $\norm{F_h(y_h,p_h)}_{Z}$ is small, where for $z_h \in V_h^*$ we define $\norm{z_h}_{Z} := \norm{w_h}_{H_0^1(\Omega)}$ with
\[
w_h \in V_h, \quad (\nabla w_h,\nabla v_h)=\langle z_h,v_h \rangle_{V_h^*} \quad \forall v_h \in V_h.
\]
Note that if $\norm{F_h(y_h,p_h)}_{Z} = 0$ then $(y_h,p_h)$ is the solution to (\ref{eqn:newtf}).
%For small $h$ this well approximates $\norm{F(y,p)}_{H^{-1}(\Omega}}$. 
The algorithm we use is stated precisely in Algorithm~\ref{alg:nm} below.

\begin{algorithm}
\caption{Newton method}
\label{alg:nm}
\begin{algorithmic}[1]
\Require $T_h, y_h^0, p_h^0$ and $\textsc{data} = (\Omega, \nu, f, a, b, I, \{y_w\}_{\omega \in I})$ \Comment{$(y_h^0, p_h^0)=(0,0)$}
\While{$\norm{F_h(y_h^k,p_h^k)}_{Z}>\delta$} \Comment{$\delta=1e-8$}
\State Compute $(\delta y_h, \delta p_h)$ by solving (\ref{eqn:sysss}): $F_h' (y_h^k,p_h^k)(\delta y_h, \delta p_h) = - F_h(y_h^k,p_h^k)$.
\State $(y_h^{k+1},p_h^{k+1}) \gets (y_h^k,p_h^k) + (\delta y_h, \delta p_h)$
\State $k \gets k+1$
\EndWhile
\State \textbf{return} $y_h^{k}, p_h^{k}$
\end{algorithmic}
\end{algorithm}

Newton type methods typically offer local superlinear convergence. We do not prove this, but we note in Section~\ref{sec:meshpde} that our algorithm is very effective in practice. On all the problems we tested it provided quadratic mesh independent convergence to the solution even with the bad initial iterate of $(0,0)$.

\subsubsection{Implementation}

%We now mention some of the finer implementational details. 
% We approximate integration against $f$ using a high order Gaussian quadrature rule, which ensures the various error estimates for $S-S_h$ and $S^*-S_h^*$ that we make use of in our proofs still hold. As a result the theorems we prove still hold with the additional forcing term. We also approximately integrate $\max(0,\cdot)$ and $\max'(0,\cdot)$ against basis functions using a high order Gaussian quadrature rule. This is not strictly necessary, as exact computation of these integrals for piecewise linear finite elements is implementable. However it would require long calculations, particularly in 3 dimensions, and we believe the error caused by this approximation to be negligible.

As we remarked above, in the case of no control constraints the first iteration of the Newton method solves (\ref{eqn:newtonfunpoint}). So rather than implementing two different numerical methods, we also use Algorithm~\ref{alg:nm} to solve the problem when there are no control constraints.

We implemented Algorithm \ref{alg:nm} in the Distributed and Unified Numerics Environment (DUNE) using DUNE-FEM (see \cite{dunegridpaperI:08, dunegridpaperII:08, dunefempaper:10}). This environment has the advantage that once an algorithm has been implemented, it is straightforward to change features of the implementation that would usually be fixed. For solving the linear systems for each iteration of the Newton method we used the biconjugate gradient stabilised method with an incomplete LU factorisation or Gauss-Seidel preconditioner.

\subsection{Exact solutions}
\label{sec:exact}

We can construct an exact solution for a simple example of the optimal control problem in dimensions 2 and 3 without control constraints. This allows us to verify our error estimates. The key fact we will use to do this is that fundamental solutions of the Laplace equation $-\Laplace y =\delta_{x'}$ are given by
\[
\begin{cases}
-\frac{1}{2 \pi} \log \abs{x-x'} + C &\quad n=2,\\
\frac{1}{4 \pi \abs{x-x'}} + C &\quad n=3.
\end{cases}
\]
So take $\Omega=B_1(0)$, the open unit ball in $\R^n$ centred at the origin, and $I=\{0\}$. Then
\begin{equation*}
p(x)=
\begin{cases}
-\frac{1}{2 \pi} \log \abs{x}(y(0)-g_0) &\quad n=2\\
\frac{1}{4 \pi} (\frac{1}{\abs{x}} -1 ) (y(0)-g_0) &\quad n=3
\end{cases}
\end{equation*}
is the unique $p$ solving (\ref{eqn:systemb}), and $u=-\frac{1}{\nu}p$ (as we have no control constraints). Note that $u$ and $p$ are unbounded, however they are still $L^2(\Omega)$ functions. To see this note that converting to polar and spherical coordinates we have
\begin{align*}
\int_\Omega (\log{\abs{x}})^2 \rd x = \int_0^{2 \pi} \int_0^1 (\log{r})^2 r \, \rd r \, \rd \theta < \infty, \\
\int_\Omega \frac{1}{\abs{x}^2} \rd x= \int_0^{2\pi} \int_0^{\pi} \int_0^1 \sin{\theta} \, \rd r \, \rd \theta \, \rd \varphi < \infty.
\end{align*}
We can now set $y$ to be any function satisfying the boundary conditions (e.g.\ $y(x)=\cos (\frac{\pi \abs{x}}{2})$), and take $f = -\Laplace y - u$. We also set $\nu=1$ and $g_0=y(0)-1$ to simplify the problem and exact solution further.
 
\subsection{2D numerical results}
\label{sec:num2d}

Motivated by the above construction take $\Omega=B_1(0)$, $A=-\Laplace$, $I=\{0\}$, $g_0=0$, $b=-a=\infty$, $\nu=1$, and 
\[
f=\frac{\pi}{4} \left( \frac{2}{\abs{x}} \sin \Big( \frac{\pi \abs{x}}{2} \Big) + \pi \cos \Big ( \frac{\pi \abs{x}}{2} \Big) \right) - \frac{1}{2 \pi} \log \abs{x}.
\]
Then the solution to the control problem is
\begin{align*}
u(x) &= -p(x) = \frac{1}{2 \pi} \log \abs{x}, \\
y(x) &= \cos \Big ( \frac{\pi \abs{x}}{2} \Big ).
\end{align*}
This solution is interesting because the control is singular (infinite) at the prescribed point $(0,0)$ but it is still an $L^2(\Omega)$ function. We solve this problem numerically using the numerical method outlined in Section~\ref{sec:nummeth}, giving Figure~\ref{fig:exact2d}. Note that the solution to the discrete problem must be bounded, even though it is approximating am unbounded function. As a result, the magnitude of the spike in $u_h$ notably increases as the triangulation is refined (but $\norm{u_h}_{L^2(\Omega)}$ is stable).

\begin{figure}[ht]
\centering
\subfigure[$y_h$]{
\includegraphics[width=0.45\textwidth]{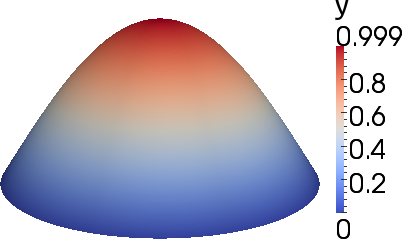}
\label{fig:yexact2d}
}
\subfigure[$u_h=-p_h$]{
\includegraphics[width=0.49\textwidth]{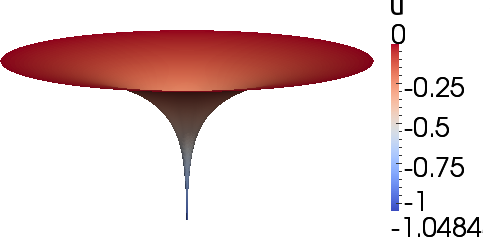}
\label{fig:pexact2d} 
}
\caption[Solution to a PDE point control problem ($n=2$, no control constraints).]{The radially symmetric solution to our 2D problem with explicitly known solution.} 
\label{fig:exact2d}
\end{figure}

The computed $L^2(\Omega)$ errors are in Table~{\rm\ref{tab:n2}}, where $\norm{u-u_h}_{L^2(\Omega)}$ is approximated using a Gaussian quadrature rule of high order, and the experimental order of convergence is defined by
\[
\mathrm{EOC}_h = \frac{\log (\norm{u-u_{h/2}}_{L^2(\Omega)}/\norm{u-u_h}_{L^2(\Omega)})}{\log 2}. 
\]
The data suggest order $h$ convergence for this problem, which agrees
with the estimate we proved in Theorem~\ref{thm:main2}
\begin{table}
\tblcaption{EOCs for the 2D problem with explicitly known solution (see Figure~\ref{fig:exact2d}).}
{%
\begin{tabular}{@{}cccc@{}}
\tblhead{\#DoFs & $h$ & $\norm{u-u_h}_{L^2(\Omega)}$ & $\text{EOC}_h$} 
% BEGIN RECEIVE ORGTBL n2
0.5 & 25 & 0.03258 & - \\
0.25 & 81 & 0.0160362 & 1.0226543 \\
0.125 & 289 & 0.00787259 & 1.0264221 \\
0.0625 & 1089 & 0.00389451 & 1.0153965 \\
0.03125 & 4225 & 0.00193778 & 1.0070370 \\
0.015625 & 16641 & 0.000966977 & 1.0028513 \\
0.0078125 & 66049 & 0.00048313 & 1.0010701
% END RECEIVE ORGTBL n2
\lastline
\end{tabular}
}
\label{tab:n2}
\end{table}
%
\begin{comment}
 #+ORGTBL: SEND n2 orgtbl-to-latex :splice t :skip 1
|         h |  DoFs |       Error |       EOC |
|-----------+-------+-------------+-----------|
|       0.5 |    25 |     0.03258 |         - |
|      0.25 |    81 |   0.0160362 | 1.0226543 |
|     0.125 |   289 |  0.00787259 | 1.0264221 |
|    0.0625 |  1089 |  0.00389451 | 1.0153965 |
|   0.03125 |  4225 |  0.00193778 | 1.0070370 |
|  0.015625 | 16641 | 0.000966977 | 1.0028513 |
| 0.0078125 | 66049 |  0.00048313 | 1.0010701 |
#+TBLFM: @2$4=0::$4=log(@-1$3/$3)/log(@-1$1/$1)
\end{comment}

The solution of a more interesting problem including control constraints and more evaluation points is shown on the left hand side of Figure~\ref{fig:eqn2d}. It appears that $p_h$ is approximating an unbounded $p$, though $\norm{p_h}_{L^2(\Omega)}$ is bounded. However $u$ is certainly bounded due to the control constraints. We do not know the exact solution to this problem so we cannot calculate the error exactly. However we can calculate an approximate order of convergence by comparing to the solution on a very fine triangulation i.e.\ $\tilde{u}=u_{h_{\text{fine}}}$ with $h_{\text{fine}}=0.00276214$, which corresponds to 263169 DOFs. So we instead compute 
\begin{equation}
\label{eqn:approxeoc}
\mathrm{EOC}_h = \frac{\log (\norm{\tilde{u}-u_{h/2}}_{L^2(\Omega)}/\norm{\tilde{u}-u_h}_{L^2(\Omega)})}{\log 2}
\end{equation}
for $h \gg h_{\mathrm{fine}}$. We ensure that the fine triangulation is a refinement of the coarser triangulations, so the $L^2(\Omega)$ errors can be computed accurately using a high order Gaussian quadrature. These approximate EOCs can be seen in Table~\ref{tab:n2cc}. They agree with the error estimate we proved for the case of active control constraints in Theorem~\ref{thm:main2}. The slight increase in the EOC for the smallest value of $h$ is expected as we are computing the error against a discrete solution and not the true solution.

\begin{figure}[ht]
\centering
\subfigure[$y_h$ with $b=-a=10$.]{
\includegraphics[width=0.42\textwidth]{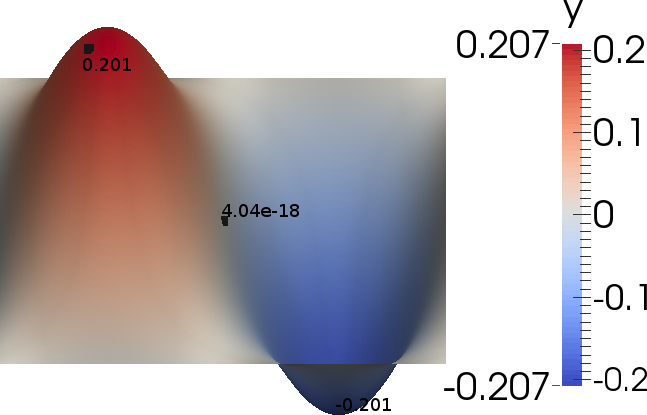}
\label{fig:eqn2dy}
}
\subfigure[$y_h$ with $b=-a=\infty$.]{
\includegraphics[width=0.43\textwidth]{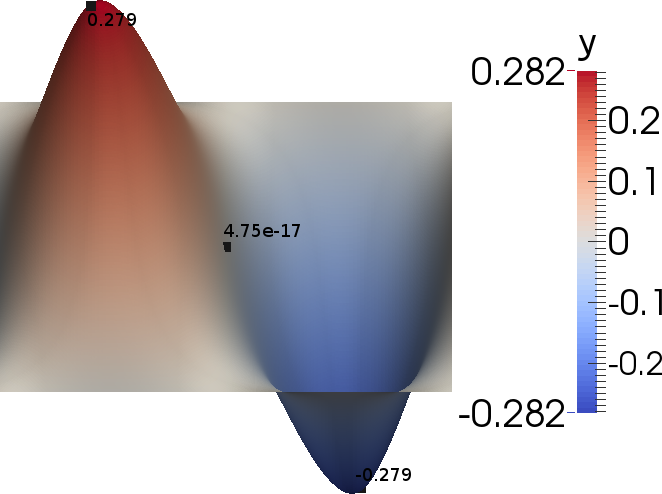}
\label{fig:eqn2dy2}
}
\subfigure[$p_h$ with $b=-a=10$.]{
\includegraphics[width=0.42\textwidth]{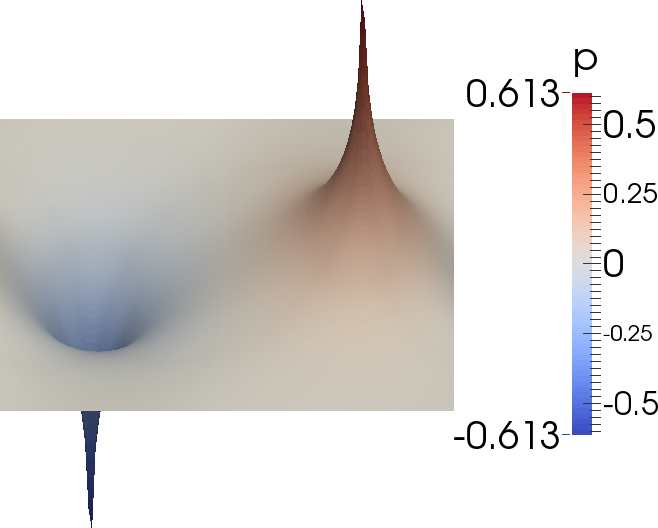}
\label{fig:eqn2dp}
}
\subfigure[$p_h$ with $b=-a=\infty$.]{
\includegraphics[width=0.42\textwidth]{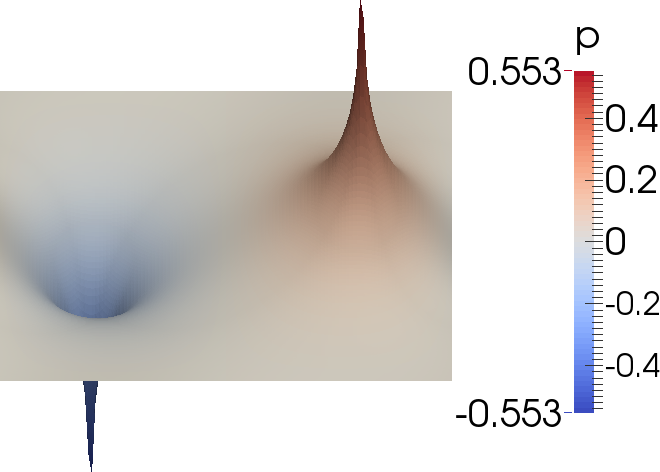}
\label{fig:eqn2dp2}
}
\subfigure[$u_h=\mathbb{P}_{[a,b]}(-\frac{1}{\nu}p_h)$ with $b=-a=10$.]{
\includegraphics[width=0.42\textwidth]{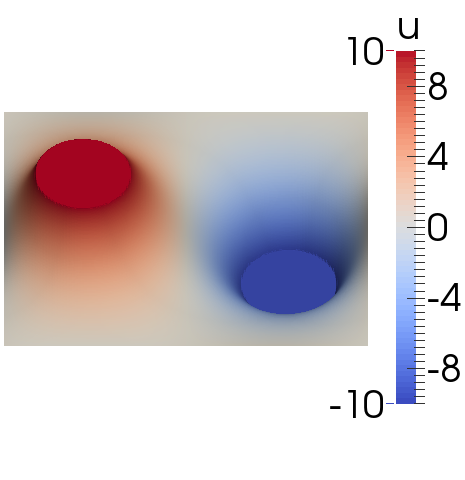}
\label{fig:eqn2du}
}
\subfigure[$u_h=\mathbb{P}_{[a,b]}(-\frac{1}{\nu}p_h)$ with $b=-a=\infty$.]{
\includegraphics[width=0.42\textwidth]{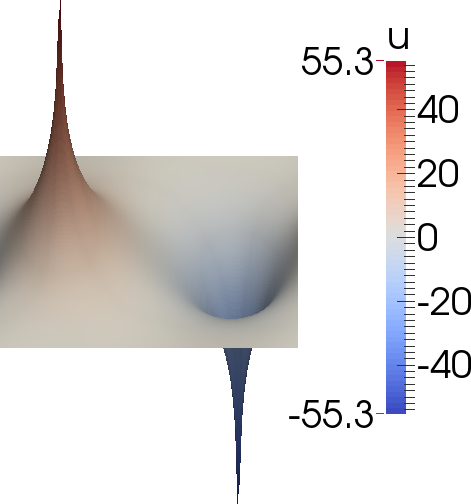}
\label{fig:eqn2du2}
}
\caption[Solution to a PDE point control problem ($n=2$, with and without control constraints).]{Solutions for $\Omega = (0,1)^2$, $A=-\Laplace$, $f=0$, $I=\{(0.2,0.5),$ $(0.5,0.5),$ $(0.8,0.5)\}$, $y_{(0.2,0.5)}=1$, $y_{(0.5,0.5)}=0$, $y_{(0.8,0.5)}=-1$, and $\nu=1e-2$. The solution on the right has $b=-a=10$ and the solution on the left has no control constraints ($b=-a=\infty$). The scale on figures that are side by side is the same. The black dots mark the locations of the points in $I$ and the numbers give the value of $y_h$ at these points.}
\label{fig:eqn2d}
\end{figure}
 
\begin{table}
\tblcaption{EOCs for the 2D problem on the left hand side of Figure~\ref{fig:eqn2d}, which has control constraints.}
{%
\begin{tabular}{@{}cccc@{}}
\tblhead{$h$ & \#DoFs & $\norm{u-u_h}_{L^2(\Omega)}$ & $\text{EOC}_h$}
% BEGIN RECEIVE ORGTBL n2con
0.353553 & 25 & 2.8881 & - \\
0.176777 & 81 & 1.51039 & 0.93520339 \\
0.0883883 & 289 & 0.80295 & 0.91153608 \\
0.0441942 & 1089 & 0.409627 & 0.97100093 \\
0.0220971 & 4225 & 0.205786 & 0.99316598 \\
0.0110485 & 16641 & 0.100486 & 1.0341436
%0.00552427 & 66049 & 0.0450935 & 1.1560092 \\
% END RECEIVE ORGTBL n2con
\lastline
\end{tabular}
}%
\label{tab:n2cc}
\end{table}
%
\begin{comment}
 #+ORGTBL: SEND n2con orgtbl-to-latex :splice t :skip 1
|          h |  DoFs |     Error |        EOC |
|------------+-------+-----------+------------|
|   0.353553 |    25 |    2.8881 |          - |
|   0.176777 |    81 |   1.51039 | 0.93520339 |
|  0.0883883 |   289 |   0.80295 | 0.91153608 |
|  0.0441942 |  1089 |  0.409627 | 0.97100093 |
|  0.0220971 |  4225 |  0.205786 | 0.99316598 |
|  0.0110485 | 16641 |  0.100486 |  1.0341436 |
| 0.00552427 | 66049 | 0.0450935 |  1.1560092 |
#+TBLFM: @2$4=0::$4=log(@-1$3/$3)/log(@-1$1/$1)
\end{comment}
 
On the right hand side of Figure~\ref{fig:eqn2d} we have the solution of the 2D problem we just considered but without the control constraints. We observe that this allows the state to get slightly closer to the prescribed values. In order to get closer still we would need to decrease $\nu$. Figure~\ref{fig:eqn2dcomp} shows a more interesting example with $\nu=1e-4$ (i.e.\ very small). As a result the state takes values very close to the prescribed values, and overshoots the value $1$ on parts of the domain in order to achieve this.

\begin{figure}[ht]
\centering
\subfigure[$y_h$]{
\includegraphics[width=0.47\textwidth]{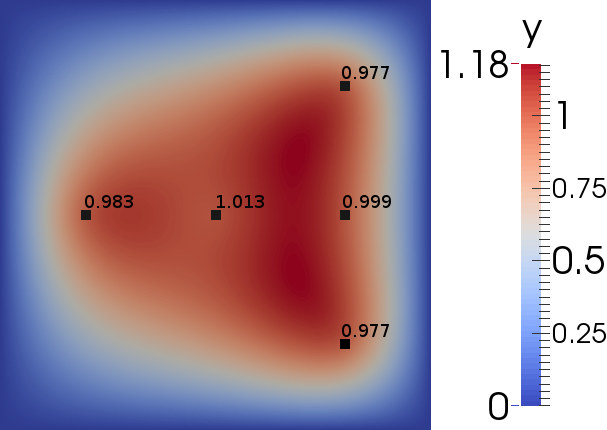}
\label{fig:eqn2dcompy}
}
\subfigure[$y_h$]{
\includegraphics[width=0.47\textwidth]{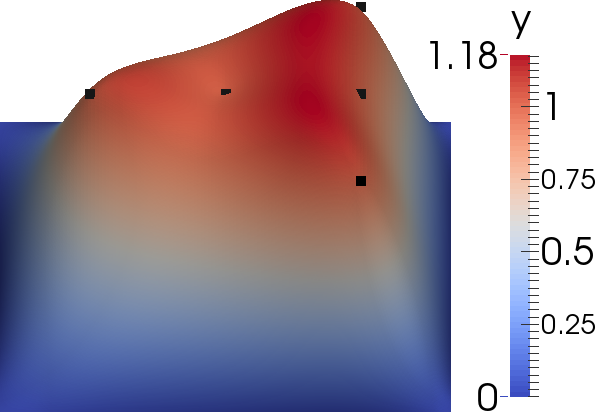}
\label{fig:eqn2dcompy2}
}
\subfigure[$p_h$]{
\includegraphics[width=0.51\textwidth]{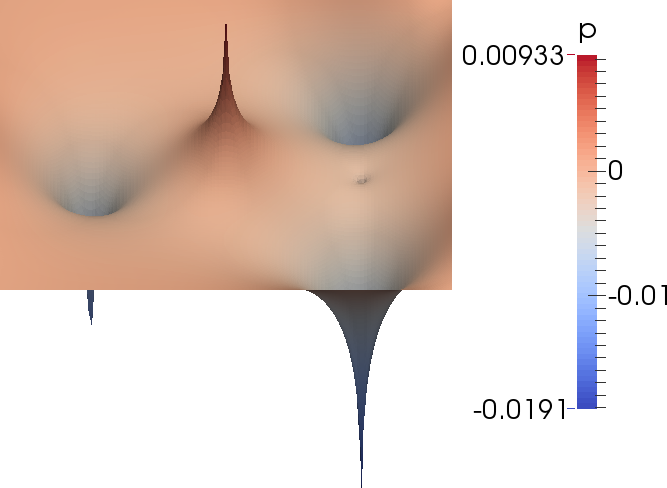}
\label{fig:eqn2dcompp}
}
\subfigure[$u_h = -\frac{1}{\nu} p_h$]{
\includegraphics[width=0.44\textwidth]{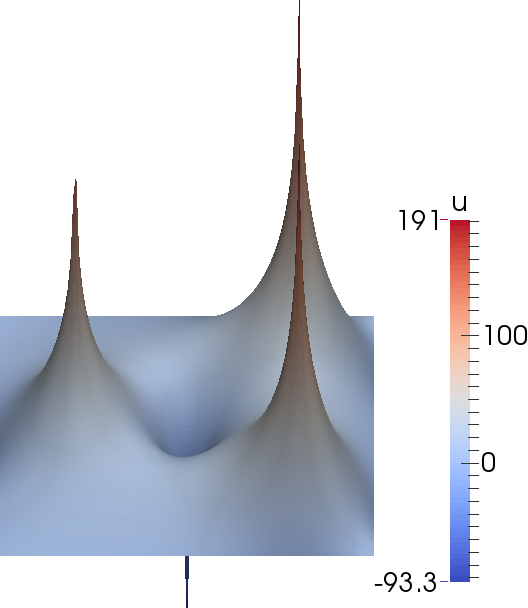}
\label{fig:eqn2dcompu}
}
\caption[Solution to a PDE point control problem ($n=2$, no control constraints).]{Solution to a more interesting example with $\Omega = (0,1)^2$, $A=-\Laplace$, $f=0$, $I=\{(0.2,0.5),$ $(0.5,0.5),$ $(0.8,0.2),$ $(0.8,0.5),$ $(0.8,0.8)\}$, $g_\omega=1$ for all $\omega \in I$, $\nu=1e-4$, and $b=-a=\infty$.\label{fig:eqn2dcomp}}
\end{figure}

\subsection{3D numerical results}
\label{sec:num3d}

Similarly take $\Omega=B_1(0)$, $A=-\Laplace$, $I=\{0\}$, $g_0=0$, $b=-a=\infty$, $\nu=1$, and 
\[
f=\frac{\pi}{4} \left( \frac{4}{\abs{x}} \sin \Big( \frac{\pi \abs{x}}{2} \Big) + \pi \cos \Big ( \frac{\pi \abs{x}}{2} \Big) \right ) + \frac{1}{4 \pi} \Big (\frac{1}{\abs{x}}-1 \Big ).
\]
Then the solution to the control problem is
\begin{align*}
u(x) &= -p(x) = -\frac{1}{4 \pi} \Big (\frac{1}{\abs{x}} -1 \Big ), \\
y(x) &= \cos \Big ( \frac{\pi \abs{x}}{2} \Big ).
\end{align*}

\begin{figure}[ht]
\centering
\subfigure[$y_h$]{
\includegraphics[width=0.47\textwidth]{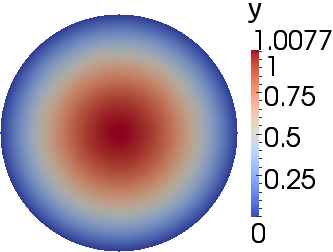}
\label{fig:y3d}
}
\subfigure[$u_h=-p_h$]{
\includegraphics[width=0.45\textwidth]{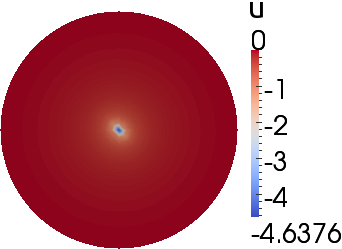}
\label{fig:p3d}
}
\caption[Solution to a PDE point control problem ($n=3$, no control constraints).]{A slice passing through the origin of the radially symmetric solution to our 3D problem with explicitly known solution.}
\label{fig:3d}
\end{figure}

This solution can be seen in Figure~\ref{fig:3d}. We observe order $\sqrt{h}$ convergence (see Table~\ref{tab:n3}), which again agrees with the estimate we proved in Theorem~\ref{thm:main2}.

\begin{table}
\tblcaption{EOCs to our 3D problem with explicitly known solution (see Figure~\ref{fig:3d}).}
{%
\begin{tabular}{@{}cccc@{}}
\tblhead{$h$ & \#DoFs & $\norm{u-u_h}_{L^2(\Omega)}$ & $\text{EOC}_h$}
% BEGIN RECEIVE ORGTBL n3
1 & 27 & 0.103658 & - \\
0.5 & 125 & 0.0719594 & 0.52657640 \\
0.25 & 729 & 0.0474726 & 0.60008809 \\
0.125 & 4913 & 0.0322929 & 0.55587806 \\
0.0625 & 35937 & 0.0225399 & 0.51873589
% END RECEIVE ORGTBL n3
\lastline
\end{tabular}
}%
\label{tab:n3}
\end{table}
%
\begin{comment}
 #+ORGTBL: SEND n3 orgtbl-to-latex :splice t :skip 1
|      h |  DoFs |     Error |        EOC |
|--------+-------+-----------+------------|
|      1 |    27 |  0.103658 |          - |
|    0.5 |   125 | 0.0719594 | 0.52657640 |
|   0.25 |   729 | 0.0474726 | 0.60008809 |
|  0.125 |  4913 | 0.0322929 | 0.55587806 |
| 0.0625 | 35937 | 0.0225399 | 0.51873589 |
#+TBLFM: @2$4=0::$4=log(@-1$3/$3)/log(@-1$1/$1)
\end{comment}

\subsection{Mesh independence}
\label{sec:meshpde}

We finish by justifying the effectiveness of our numerical method. When we have no control constraints the problem is linear and the Newton method always finds the exact solution in a single iteration. When we have control constraints the problem is nonlinear and  we still have good mesh independence properties; the number of Newton iterations needed for convergence does not increase as $h$ is decreased. See Table \ref{tab:meshitspde} for the number of Newton iterations needed to solve the control constrained example from Figure~\ref{fig:eqn2d} using the initial iterate $(0,0)$.

We also observe quadratic convergence of the Newton method on average. See Table \ref{tab:newtconpde} for the residuals of the Newton method, again for the control constrained example from Figure~\ref{fig:eqn2d}. In the table
\begin{equation}
\label{eqn:eocnewt}
\mathrm{EOC}_k := \frac{\log(\delta_{k+1}/\delta_k)}{\log(\delta_k/\delta_{k-1})}, \quad \delta_k := \norm{F_h(y_h^k,p_h^k)}_{H^{-1}(\Omega)}.
\end{equation}

\begin{table}
\parbox{.4\linewidth}{
\centering
\tblcaption{Number of iterations of Newton method.}
{%
\begin{tabular}{@{}cc@{}}
\tblhead{$h$ & \# iterations} 
% BEGIN RECEIVE ORGTBL its
0.0883883 & 3 \\
0.0441942 & 3 \\
0.0220971 & 3 \\
0.0110485 & 3 \\
0.00552427 & 3
% END RECEIVE ORGTBL its
\lastline
\end{tabular}
}%
\label{tab:meshitspde}
}
%
\begin{comment}
 #+ORGTBL: SEND its orgtbl-to-latex :splice t :skip 1
|          h | #iterations |
|  0.0883883 |          3 |
|  0.0441942 |          3 |
|  0.0220971 |          3 |
|  0.0110485 |          3 |
| 0.00552427 |          3 |
#+TBLFM:
\end{comment}
\parbox{.45\linewidth}{
\centering
\tblcaption{Convergence rate of Newton method.}
{%
\begin{tabular}{@{}ccc@{}}
\tblhead{Iteration $k$ & $\norm{F_h(u^k_h)}_{Z}$ & $\mathrm{EOC}_k$}
% BEGIN RECEIVE ORGTBL newt
0 & $0.00285272$ & 0 \\
1 & $4.38339\times 10^{-5}$ & 0.85936125 \\
2 & $1.21172\times 10^{-6}$ & 2.4577166 \\
3 & $1.79175\times 10^{-10}$ & 0 
% END RECEIVE ORGTBL newt
\lastline
\end{tabular}
}%
\label{tab:newtconpde}
}
\end{table}

\bibliographystyle{IMANUM-BIB}
\bibliography{library}

\end{document}